\documentclass[12pt,a4paper]{article} 
\usepackage[colorlinks]{hyperref}
\usepackage{amsmath}
\usepackage{amsthm}
\usepackage{amsfonts}
\usepackage{amssymb}
\usepackage{mathabx}
\usepackage{bussproofs}
\usepackage{lineno}

\usepackage{stmaryrd}
\usepackage{euscript} 
\usepackage{latexsym}
\usepackage{mathrsfs}
\usepackage[all]{xy}
\usepackage{dsfont}
\usepackage{proof}

\usepackage{algorithm}
\usepackage{algpseudocode}

\usepackage{graphicx}
\usepackage{relsize}

\usepackage{url}
\usepackage{enumitem}
\usepackage{xcolor}
\usepackage{colortbl}
\usepackage{colonequals}
\usepackage[all]{xy}
\usepackage{tikz}
\usetikzlibrary{positioning,arrows,calc}

\usepackage{graphics} % baraye resize e table
\usepackage{mathrsfs}
\usepackage{xcolor}
\theoremstyle{plain}

\usepackage{booktabs,array}
\newcolumntype{P}[1]{>{\raggedright\arraybackslash}p{\dimexpr#1\linewidth-2\tabcolsep}}

\theoremstyle{definition} 
\newtheorem{theorem}{Theorem}[section]

\newtheorem{lemma}[theorem]{Lemma}
\newtheorem*{lemmm}{Lemma}

\newtheorem{corollary}[theorem]{Corollary}
\newtheorem{problem}[theorem]{Problem}
\newtheorem*{theoremm}{Theorem}
\theoremstyle{definition}
\newtheorem{definition}[theorem]{Definition}
\newtheorem*{definitionm}{Definition}

\usepackage{amssymb}
\usepackage{amsmath}
\usepackage{amsthm}
\usepackage{amsfonts}
\usepackage{amssymb}
\usepackage{tikz}

\usepackage{graphicx} % Allows including images
\usepackage{booktabs} % Allows the use of \toprule, \midrule and \bottomrule in tables
\usepackage{bussproofs}
\usepackage{mathabx}

\graphicspath{ {./images/} }

\def\mod{\text{mod}\;}
\def\NumOne{\text{NumOnes}}
\def\max{\text{max}}
\def\int{\mathrm{int}}
\def\log{\text{log}\,}

\def\S{\mathrm{S}}
\def\PV{\mathrm{PV}}
\def\iWPHP{\mathrm{iWPHP}}

\usepackage{bussproofs}
\usepackage{amssymb} 
\usepackage{amsmath} 
\usepackage{stmaryrd}
\usepackage{euscript} 
\usepackage{latexsym}
\usepackage{mathrsfs}
\usepackage[all]{xy}
\usepackage{dsfont}
\usepackage{proof}

\begin{document}

\title{Feasibility of Primality in Bounded Arithmetic} 
\author{Raheleh Jalali\footnote{\texttt{rahele.jalali@gmail.com}}, Ond\v{r}ej Je\v{z}il\footnote{\texttt{ondrej.jezil@email.cz}}\\
\small{Department of Computer Science, University of Bath$^*$ }\\
\small{Faculty of Mathematics and Physics, Charles University$^\dagger$}%\footnote{Support by the Netherlands Organisation for Scientific Research under grant 639.073.807 is gratefully acknowledged.}\\ \small{Utrecht University}
}

\newcommand{\BASIC}{\text{BASIC}}
\newcommand{\NN}{\mathbb{N}}
\newcommand{\LBASIC}{L_{\BASIC}}
\newcommand{\val}{\text{val}}
\newcommand{\xgcd}{\text{xgcd}}
\newcommand{\lcm}{\text{lcm}}
\newcommand{\LCM}{\text{LCM}}
\newcommand{\Binom}{\text{Binom}}
\newcommand{\ord}{\text{ord}}
\newcommand{\ORD}{\text{ORD}}
\newcommand{\Pow}{\text{Pow}}
\newcommand{\Poly}{\text{Poly}}
\newcommand{\findval}{\text{findval}}
\newcommand{\sigb}[1]{\Sigma^b_{#1}}
\newcommand{\md}[1]{\ (\text{mod }{#1})}
\newcommand{\mdl}{\text{ mod }}
\newcommand{\floor}[1]{\left \lfloor {#1} \right\rfloor}
\newcommand{\flr}[1]{\lfloor {#1} \rfloor}
\newcommand{\abs}[1]{{\lvert {#1} \rvert}}
\renewcommand{\div}[2]{\floor{#1/#2}}
\newcommand{\half}[1]{\floor{#1/2}}
\newcommand{\Fact}[1]{\text{Fact}({#1})}
\newcommand{\FactD}[1]{\text{Fact}({#1})}
\newcommand{\Log}{\text{Log}}
\newcommand{\mult}{\textsc{mult}}
\newcommand{\Prime}{\text{Prime}}
\renewcommand{\mod}[2]{#1\text{ mod }#2}
\newcommand{\delb}[1]{\Delta^b_{#1}}
\renewcommand{\S}[1]{S^i_2}
\newcommand{\AKSPrime}{\textsc{AKSPrime}}
\newcommand{\AKSCorrect}{\textsc{AKSCorrect}}
\newcommand{\PerfectPower}{\textsc{PerfectPower}}
\newcommand{\ct}{\text{count}}
\newcommand{\Primes}{\textsc{Primes}}
\renewcommand{\P}{\textbf{P}}
\newcommand{\VTC}{\text{VTC}}
\newcommand{\card}{\text{card}}
\newcommand{\TC}{\text{TC}}
\newcommand{\NP}{\text{NP}}
\newcommand{\AC}{\text{AC}}
\newcommand{\GFLT}{\text{GFLT}}
\newcommand{\PHP}{\text{PHP}}
\newcommand{\IMUL}{\text{IMUL}}
\newcommand{\Tot}{\textsc{Tot}}
\newcommand{\bit}{\textrm{bit}}
\newcommand{\Cyc}{\textsc{Cyc}}
\newcommand{\DLB}{\mathrm{RUB}}
\newcommand{\sblm}{\Sigma^b_1\text{-LMIN}}
\newcommand{\sblma}{\Sigma^b_1\text{-LMAX}}
\newcommand{\X}{\mathcal{X}}
\newcommand{\Y}{\mathcal{Y}}
\newcommand{\G}{\mathcal{G}}
\newcommand{\R}{\text{R}}
\newcommand{\F}{\mathbb{F}}
\newcommand{\ZZ}{\mathbb{Z}}
\date{\today}
\maketitle
\begin{abstract}
We prove the correctness of the AKS algorithm \cite{AKS} within the bounded arithmetic theory $T^\ct_2$ or, equivalently, the first-order consequences of the theory $\VTC^0$ expanded by the smash function, which we denote by $\VTC^0_2$. Our approach initially demonstrates the correctness within the theory $S^1_2 + \iWPHP$ augmented by two algebraic axioms and then show that they are provable in $\VTC^0_2$. The two axioms are: a generalized version of Fermat's Little Theorem and an~axiom adding a new function symbol which injectively maps roots of polynomials over a definable finite field to numbers bounded by the degree of the given polynomial. To obtain our main result, we also give new formalizations of parts of number theory and algebra: 
\begin{itemize}
 \item In $\PV_1$: We formalize Legendre's Formula on the prime factorization of $n!$, key properties of the Combinatorial Number System and the existence of cyclotomic polynomials over the finite fields $\ZZ/p$.   
 \item In $S^1_2$: We prove the inequality $\lcm(1,\dots, 2n) \geq 2^n$.
 \item In $\VTC^0$: We verify the correctness of the Kung--Sieveking algorithm for polynomial division. 
 \end{itemize}

\end{abstract}

\newpage
\tableofcontents

\newpage
\section{Introduction}

In 2002 the first deterministic polynomial-time algorithm for primality testing was found --- the AKS algorithm (named after its creators Agrawal, Kayal, and Saxena)~\cite{AKS}. The containment $\Primes \in \P$ can be interpreted as saying: ``Primality is a \emph{feasible} property''. The general question we are treating in this work is: 
\begin{itemize}
    \item How feasible can a proof of the statement $\Primes \in \P$ be?
\end{itemize} 
More concretely, since the AKS algorithm is currently the only known primality testing algorithm unconditionally running in deterministic polynomial time, the particular way we interpret this question is to ask: 
\begin{itemize}
    \item How feasible can a proof of the correctness of the AKS algorithm be?
\end{itemize}

The feasibility of a proof here is captured by concepts from proof complexity, namely from the field of bounded arithmetic. The theories of bounded arithmetic, which are in depth treated in~(\cite{krajicek1995bounded}, \cite{cook2010logical}), are comparatively weak subtheories of Peano arithmetic. Each such theory, in some precise technical sense, corresponds to reasoning with concepts from a concrete complexity class --- the theory corresponding to polynomial-time (or \emph{feasible}) reasoning being $\PV_1$. There is already a rich body of results formalizing the correctness of various algorithms and parts of complexity theory in $\PV_1$ and its extensions (\cite{pich2015logical},\cite{muller2020feasibly},\cite{jerabek2022iterated},\cite{Igor},\cite{gaysin2024proof}), the main motivation being:

\begin{itemize}
    \item \emph{Bounded reverse mathematics}: Formalizing known mathematics in such theories provides insight into the intrinsic power of these theories and into the feasibility of the underlying constructions. For more details, see \cite{Nguyen,chen2024reverse,Igor25}.
    \item \emph{Propositional translations}: Provability of sharply bounded formulas in a bounded arithmetic theory $T$ %gives upper bounds on propositional proof system $P$ corresponding to the theory.
    gives a specific proof system corresponding to $T$ in which all propositional tautologies expressing the formula for inputs of various lengths can be efficiently proved.
    \item \emph{Witnessing results}: Provability of a suitable existential formula in a bounded arithmetic theory $T$ can give an algorithm computing the witness in the complexity class $C$ corresponding to the theory $T$.
\end{itemize}
Regarding primality testing, Jeřábek has proved the correctness of the Rabin--Miller primality test in $S^2_2 + \iWPHP(\PV) + \PHP(\PV)$~\cite{emildual,EmilAbelian}.

By the correctness of the AKS algorithm, we mean the statement 
\[(\forall x)(\Prime(x) \leftrightarrow  \AKSPrime(x)),\]
where $\Prime(x)$ is the usual $\Pi^b_1$-formula defining primality and $\AKSPrime(x)$ is a predicate stating that the AKS algorithm claims that the number $x$ is a prime. Herbrand's theorem readily implies that if
\[\PV_1 \vdash (\forall x)(\Prime(x) \to \AKSPrime(x)),\]
then there is a polynomial-time algorithm for factoring integers, which seems utmost unlikely. This conditional unprovability result is not the only obstacle in formalizing the proof. While the original proof is relatively elementary, it still involves number-theoretic and algebraic statements, which were so far untreated in the context of bounded arithmetic.  We show that the correctness of the AKS algorithm for primality testing can be proved in the bounded arithmetic theory $T_2^{\ct}$ by first showing that it can be proved in $S^1_2 + \iWPHP$ enriched by two algebraic axioms. The first axiom is \emph{Generalized Fermat's Little theorem} (Definition \ref{definitiongflt}), which asserts that the polynomials $(\X+a)^p$ and $(\X^p+a)$ are congruent modulo any polynomial $\X^r-1$ modulo $p$, where $r$ is of logarithmic size. The second axiom is called \emph{Root Upper Bound}, $\DLB$ (Section \ref{subsecdlb}): it essentially axiomatizes a new function symbol, which provides an injective function from the roots of the polynomial $f$ into the set $\{1,\dots,\deg f\}$, which we then allow to appear in the formulas of the schemes axiomizing the base theories $\PV_1$, $S^1_2$ and $T_2$. As additional background on formalizing complexity-theoretic statements in bounded arithmetic, it is worth noting that \cite{Atserias} formalizes the Schwartz-Zippel Lemma\footnote{Schwartz-Zippel Lemma states that if a low-degree multivariate polynomial with coefficients
in a field is not zero everywhere in the field, then it has few roots on every
finite subcube of the field.} in $S^1_2$. They also formalize $\mathrm{FTA}_{\le}$, half of Fundamental Theorem of Algebra, in $S^1_2$ \cite[Lemma 3.2.]{Atserias}. 
Let us compare $\mathrm{FTA}_{\le}$ and our axiom  $\DLB$. We see that $\mathrm{FTA}_{\le}$ works for small-degree univariate polynomials in small-size fields and constructively lists all roots.  $\DLB$ handles arbitrary-degree sparse polynomials and provides an injective indexing of roots. $\mathrm{FTA}_{\le}$ is stronger constructively within its limited domain, but  $\DLB$ applies to cases $\mathrm{FTA}_{\le}$ cannot handle, making it essential for general sparse polynomial reasoning.

Along the way, we prove in $\PV_1$ Legendre's theorem~(Theorem~\ref{thrmlegendre}), the existence of \emph{cyclotomic extensions} of finite fields $\ZZ/p$~(Theorem~\ref{cycltomexists}) and the correctness of the combinatorial number system~(Lemma~\ref{lem: injective function}), in $S^1_2$ the bound $\lcm(1,\dots,2n) \geq 2^n$~(Corollary \ref{corollarylcmbound}),  and in $\VTC^0$ the correctness of the Kung--Sieveking algorithm~(Lemma~\ref{lemmavtcpolydiv}) for polynomial division, which could be of independent interest.

Another notable point is that the formalized proofs closely follow the original AKS proofs, preserving their core combinatorial and number-theoretic structure. This shows that bounded arithmetic can naturally capture these fundamental results without requiring major changes to the underlying arguments.
\section{Preliminaries}

\subsection{Bounded arithmetic and the theories in question}

In this section, we recall basic facts about the theories relevant to our formalization. We refer the reader interested in a comprehensive treatment of bounded arithmetic to~\cite{krajicek1995bounded}.

%\textcolor{red}{Small Roman letters $x,y, \ldots$ denote numbers. the letters $f,g,h ...$ denote ... small Greek letters ... Curly capital letters $X,Y$ denote the set sort, curly capital letters $\mathcal{X}, \mathcal{Y}$ denote the variable for ...}

The minimal language we consider is Buss' language $L_{S_2}$~\cite{buss1985} used in the original definition of the theory $S^1_2$ (see below). The language $L_{S_2}$ consists of the non-logical symbols $\{0,S,+,\cdot, \leq, \#, \floor{x/2},  \abs{x}\}$ and all the usual logical symbols. The symbols $0, S,+,\cdot, \leq$ are the zero constant, the successor function, addition, multiplication,
and the less-than-or-equal-to relation. The intended meaning of $\floor{x/2}$ is to divide by two and round down, and $|x|$ is $\lceil \log_2 (x+1) \rceil$, which is the length of the binary representation of $x$. We define $x \# y=2^{|x| \cdot |y|}$. The behaviour of these symbols is axiomatized by the 32 universal axioms of $\BASIC$ (see \cite[Section 2.2]{buss1985}). As we will only consider logarithms base 2, from now on, by $\log x$, we mean $\log_2 (x)$. Denote $S(0)$ by $1$. The \emph{$L_{S_2}$-terms} and \emph{$L_{S_2}$-formulas} are defined as usual. %When there is no risk of confusion, we will simply refer to them as \emph{terms} and \emph{formulas}.

A \emph{bounded quantification} of the free variable $x$ in an $L_{S_2}$-formula $\varphi(x)$ is either of the following
\begin{align*}
    &(\forall x \leq t)(\varphi(x))\equiv(\forall x)(x\leq t \to \varphi(x))\\
    &(\exists x \leq t)(\varphi(x))\equiv(\exists x)(x\leq t \land \varphi(x)),
\end{align*}
where $t$ is an $L_{S_2}$-term that does not contain $x$. We call the quantifiers of the above form \emph{bounded}. We say a quantifier is \emph{sharply bounded} if it is bounded and the term $t$ is of the form $\abs{s}$ for some term $s$. A \emph{(sharply) bounded formula} in the language $L_{S_2}$ is simply a formula in which every quantifier is (sharply) bounded. The set of all bounded formulas is denoted~$\Sigma^b_\infty$. An  $L_{S_2}$-formula is $\Sigma^b_0=\Pi^b_0$ if all its quantifiers are sharply bounded. An $L_{S_2}$-formula is $\Sigma^b_1$ (resp. $\Pi^b_1$) if it is constructed from sharply bounded formulas using conjunction, disjunction, sharply bounded quantifiers and existential (resp. universal) bounded quantifiers. %Denote $\Sigma^b_1 \cap \Pi^b_1$ by $\Delta^b_1$. (TODO: Do we need this? $\Delta^b_1$-definability is mentioned, but definition of that does not use the intersection of $\Sigma^b_1$ and $\Pi^b_1$.)

The theory $S^1_2$ is axiomatized over BASIC by the polynomial induction schema $\Sigma^b_1$-PIND for any $\Sigma^b_1$-formula $\varphi$:
\[
\varphi(0) \wedge \forall x (\varphi(\floor{x/2}) \to \varphi(x)) \to \varphi(a),
\]
or alternatively by the length minimization schema $\Sigma^b_1$-LMIN:
\[\varphi(a) \to (\exists x\leq a)(\forall y \leq a)(\varphi(x) \land (\abs{y}<\abs{x} \to \lnot \varphi(y))\]
for any $\Sigma^b_1$-formula $\varphi$, or by the length maximization schema $\Sigma^b_1$-LMAX:
\[
\varphi(0) \to (\exists x \leq a)(\forall y \leq a)(\varphi(x)\land (\abs{x}<\abs{y} \to \lnot \varphi(y)))
\]
for any $\Sigma^b_1$-formula $\varphi$.
%if it is equivalent over predicate calculus to a formula formed from open formulas by conjunction, disjunction, sharply bounded quantification and bounded existential quantification. We say $L_{S_2}$-formula is $\Pi^b_1$ if it is equivalent over predicate calculus to a negation of $\Sigma^b_1$ formula.
The theory $T_2$ is axiomatized over BASIC by the induction schema $\Sigma^b_{\infty}$-IND:
\[
\varphi(0) \wedge \forall x (\varphi(x) \to \varphi(x + 1)) \to \forall x \varphi(x).
\]
for any $\Sigma^b_{\infty}$-formula, or equivalently bounded formula $\varphi$.

Throughout this paper, we shall use the richer language of Cook's theory $\PV$~\cite{cook1975} and of the theory $\PV_1$~\cite{krajicek1991bounded}, which contains function symbols for every polynomial-time algorithm. In particular, this language includes $L_{S_2}$. With a slight abuse of notation, we also use $\PV$ to refer to this richer language.
We will simply refer to these function symbols for polynomial-time algorithms as \emph{$\PV$-symbols}. We shall sometimes treat them as predicates, which we call \emph{$\PV$-predicates}. The intended meaning of a $\PV$-predicate is that a corresponding $\PV$-symbol for the characteristic function of the predicate outputs $1$. The definitions of $\Sigma^b_1$, $\Pi^b_1$, and $\Sigma^b_\infty$ can be extended to this new language by allowing the terms and open formulas to be in $\PV$. These classes of formulas are called $\Sigma^b_1(\PV)$, $\Pi^b_1(\PV)$, and $\Sigma^b_\infty(\PV)$, and the expanded theories $S^1_2(\PV)$ and $T_2(\PV)$. 
However, with a slight abuse of the notation, we use $\Sigma^b_1$, $\Pi^b_1$, and $\Sigma^b_\infty$, $S^1_2$, and $T_2$.

The weakest theory we consider is $\PV_1$, first defined in~\cite{krajicek1991bounded} as a conservative extension of Cook's equational theory\footnote{Another theory $\PV1$ was introduced simultaneously with $\PV$ in~\cite{cook1975}, this theory uses the universal fragment of first order logic. Compare the notation to $\PV_1$ which is a first order theory.} $\PV$~\cite{cook1975}. The theory $\PV_1$, a universal first-order theory, formalizes polynomial-time reasoning in the language of $\PV$. %, which contains the defining equations for all $\PV$-symbols and open induction. 
It is shown in~\cite{krajicek1995bounded} that $\PV_1$ proves $\Sigma^b_0$-IND. We say that the formula $A$ is \emph{$\Delta^b_1$-definable} in a theory $R$ if{f} there are formulas $B \in \Sigma^b_1$ and $C \in \Pi^b_1$ such that $A \leftrightarrow B$ and $A \leftrightarrow C$ are provable in the theory $R$. Every $\PV$-function can be $\Delta^b_1$-defined in $S^1_2$, hence we can see $S^1_2$ as an extension of $\PV_1$ (see~\cite{buss1985}). Also, using Buss' witnessing theorem~\cite{buss1985}, we can see that $S^1_2$ is $\forall \Sigma^b_1$-conservative over $\PV_1$. Therefore, when we want to prove a $\forall \Sigma^b_1$-formula in $\PV_1$, we can prove it in $S^1_2$. As mentioned earlier, we assume that $S^1_2$ is in the language $\PV$. To prove the existence of prime factorization of numbers, it seems to be necessary to work in $S^1_2$ (see~\cite{jezil25prime}). %its $\forall\Sigma^b_1$-conservative extension $S^1_2$, which we will also assume to be in the language of $\PV$. %The theory $S^1_2$ is axiomatized over $\PV_1$ by the axiom scheme $\sblm$ consisting of axioms:\[\varphi(a) \to (\exists x\leq a)(\forall y \leq a)(\varphi(x) \land (\abs{y}<\abs{x} \to \lnot\varphi(y))\] for every $\varphi\in\Sigma^b_1$. An alternative axiomatization is given by the axiom scheme $\sblma$ consisting of axioms:\[\varphi(0) \to (\exists x \leq a)(\forall y \leq a)(\varphi(x)\land (\abs{x}<\abs{y} \to \lnot \varphi(y)))\] for every $\varphi\in\Sigma^b_1$.

Let us define an axiom schema, which we will use later. The injective weak pigeonhole principle $\iWPHP(\PV)$ is the axiom schema 
\[
a>0 \to \big((\exists x <2a)(f(x)\geq a) \lor (\exists x< x' < 2a)(f(x)=f(x'))\big)
\]
for each $\PV$-symbol $f$. We will introduce two additional axiom schemata, the \emph{Generalized Fermat's Little Theorem} ($\GFLT$) and the \emph{Root Upper Bound} ($\DLB$), once we have defined all the required notions. The axiom $\GFLT$ is treated in Definition~\ref{definitiongflt}, and the axiom $\DLB$ and the resulting theory $S^1_2+\iWPHP+ \DLB$ are treated in Section~\ref{subsecdlb}.

Another essential theory for us is the theory $T^\ct_2$. It consists of the theory $T_2(\PV)$ %extending $\PV$ by induction axioms for every $\varphi\in\Sigma^b_\infty$ formula 
with its language extended by a recursively defined family of counting functions for each $\Sigma^b_\infty$-formula $A$. Then, for each bounded formula, it contains the first level counting functions, and so on. We will not work with this theory directly. Instead, we will consider $\Sigma^b_\infty$-consequences of the theory $\VTC^0_2$ which are fully conservative over $T^\ct_2$~\cite[essentially Lemma 13.1.2]{krajicek1995bounded}.

Let us now define the theory $\VTC^0_2$. First, the theory $V^0_2$ is a two-sorted extension of $T_2$, where the first sort is the `\emph{number sort}' denoted by lowercase variables $x, y, \ldots$ as in $T_2$, and the new `\emph{set sort}' corresponds to binary strings whose bits are indexed by the number sort and is denoted by uppercase variables $X, Y, \ldots$. The language of $V^0_2$ is extended by an equality symbol for the set sort, the elementhood relation $x\in X$ between the number sort on the left and the set sort on the right, and the length symbol $\abs{X}$, which takes as an input an element of the set sort and outputs an element of the number sort. The intended meaning of $\abs{X}$ is the least strict upper bound on elements of $X$. There is one exception to the number sort and set sort notation: The name of an arbitrary bounded field will be denoted $F$ despite bounded fields being objects of the number sort. Note that only in the language of $V^0_2$ we have two sorts (i.e., the number sort and the set sort), and we use the uppercase and lowercase letters to distinguish between them. However, in the language of $\PV$, there is only one sort, i.e., the number sort. Therefore, when we are working in a theory in the language of $\PV$, such as $\PV_1$ or $S^1_2$, we freely use uppercase letters also to denote objects of the number sort.

A $\Sigma^B_0$-formula is a bounded formula in this new language without quantification of the set sort.
Following~\cite{jerabek2022iterated}, the theory $V^0_2$ can be axiomatized by $T_2$, the basic axioms
\begin{align*}
    &\abs{X}\neq 0 \to (\exists x)(x\in X \land \abs{X}=x+1)\\
    &x\in X \to x < \abs{X}\\
    &(\forall x)((x\in X \leftrightarrow x\in Y) \to X=Y)
\end{align*}
and the comprehension schema for $\Sigma^B_0$-formulas. That is, for every $\Sigma^B_0$-formula $\varphi$, the following is an axiom schema,
\[(\exists X\leq x)(\forall u < x)(u \in X \leftrightarrow \varphi(u))\tag{$\varphi\text{-}\mathrm{COMP}$}\]
where a bounded quantification of an element of the set sort $(\exists Y \leq t)(\dots)$ is interpreted as $(\exists Y)( \abs{Y}\leq t \land \dots )$.

Note that using sets, we can encode sequences of the elements of the number sort, $X^{(i)}$ being the $i$-th element of such encoding, where $X$ is an element of the set sort and $i$ an element of the number sort.
We obtain the theory $\VTC^0_2$ by extending $V^0_2$ by the axiom $(\mathrm{CARD})$:
\begin{align*}
(\forall n)(\forall X)(\exists Y)(&Y^{(0)}=0 \land \\&(\forall i<n )((i\not \in X \to Y^{(i+1)}=Y^{(i)}) \land\\&\:\phantom{(\forall i<n)}\: (i\in X \to Y^{(i+1)} = Y^{(i)}+1)))
\end{align*}
In essence, the theory $\VTC^0_2$ is a simple extension of the two-sorted theory $\VTC^0$ originally defined in~\cite{cook2010logical} by the smash function $\#$ for the number sort, and in our setting also by all $\PV$-symbols for the number sort. In this work, we understand the theory $\VTC^0$ as a theory whose axioms are the ones of $\VTC^0_2$ which do not contain the symbol for the smash function $\#$.

We say a formula in the language of $\VTC^0_2$ is $\Sigma^B_1$ if it is equivalent over predicate calculus to a formula that contains only bounded number sort quantifications, bounded existential set sort quantification and no negations appear before any of the existential set sort quantifications. A function is $\Sigma^B_1$-definable in $\VTC^0_2$ if its graph can be defined in the standard model using a $\Sigma^B_1$-formula $\varphi(X,Y)$ such that $\VTC^0_2 \vdash (\forall X)(\exists! Y)\varphi(X,Y)$. It can be shown~\cite[Theorem IX.3.7.]{cook2010logical} that $\VTC^0_2$ proves the comprehension axiom which we shall denote $\Sigma^B_0(\card)\text{-}\mathrm{COMP}$. This comprehension axiom allows the $\Sigma^B_0$-formulas to contain new function symbols each computing a $\Sigma^B_1$-definable function in $\VTC^0_2$, similarly it proves the induction axiom $\Sigma^B_0(\card)\text{-}\mathrm{IND}$, which allows such formulas in the induction axiom.

\subsection{Algebraic primitives in bounded arithmetic}

Let us start with the formalization of elementary number theory and polynomial arithmetic in $\PV_1$. For numbers $x$, $y$ we denote $x$ divides $y$ by $x \mid y$. There is a well-behaved $\PV$-symbol $\mod{x}{m}$ which computes the remainder of $x$ modulo $m$. The primality of a number $x$ is defined by the $\Pi^b_1$-formula $\Prime(x)$:
\[x>1 \land (\forall y\leq x)(y \mid x \to (y=1 \lor y=x)).\]

Coding of both sets and sequences of elements can be developed in $\PV_1$, here we will describe the main ideas. We use the notation $1^{(r)}$ to denote the number whose binary expansion consists of $r$-many ones, in other words, the number $r$ is of logarithmic size. Coding of sets is simple, there is a well-behaved $\PV$-symbol $\bit(x,i)$ which outputs the $i$-th bit of the number $x$, the elementhood relation $i\in x$ is then defined to be equivalent to $\bit(x,i)=1$. Regarding sequences, a sequence $( a_1,\dots, a_k),$ where $1^{(k)}$ exists, can be coded by a pair $\langle a,b \rangle$, with $a$ consisting of a number whose binary expansion is the concatenation of the binary expansions of $a_1, a_2, \dots, a_k$ and $b$ is the number whose binary expansion contains $1$ at the positions which correspond to the beginning of each $a_i$. We shall denote the pair $\langle a,b\rangle$ by $\langle a_1,\dots,a_k\rangle$. More details on coding of sequences in $\PV_1$ can be found in~\cite[Section 5.4]{krajicek1995bounded}. Sums and products over sequences can also be defined by the obvious iterative algorithm. Such an algorithm obtains a sequence of numbers and outputs the iterated sum/product as a single number. 

Let us note that sometimes we consider sets which are not neccesarily codeable in $\PV_1$ but are definable by a formula: These are just abbreviations in the meta-language where we talk about a definable predicate as if it were the set of all elements which satisfy it. As such, definable sets are not quantifiable. Of special importance are sets of the form $\{x;C(x)=1\}$ for a Boolean circuit $C$, as these can appear in the induction formulas. We always make it clear when a set is codeable as oppose to definable.

\emph{Polynomials} are defined to be sequences of coefficients such that the number of coefficients is one more than the degree of the polynomial, with the addition, multiplication and composition defined as usual, division defined using the long division algorithm, and exponentiation modulo some other polynomial defined using exponentiation by squaring. \emph{Divisibility} of polynomials is defined as usual, the notation $f\mid g$ meaning that the polynomial $f$ divides $g$. A polynomial is called \emph{irreducible} if it cannot be written as a product of two polynomials, both having positive degrees. We will sometimes call polynomials formalized like this \emph{low degree polynomials}, as their degree is bounded by a number of logarithmic size. Later in this section we will define a second formalization of polynomials in a two sorted theory, and we will call polynomials following this second formalism high degree polynomials.

The basic properties of the \emph{greatest common divisor} ($\gcd$) of numbers are essential for many arguments. The following is well-known (see~\cite{EmilAbelian}):
\begin{lemma}[$\PV_1$, Euclid's algorithm]\label{lemmapveuclid}
    There is a $\PV$-function symbol $\gcd$ such that $\PV_1$ proves
    \begin{align*}
        &\gcd(x,y) \mid x\\
        &\gcd(x,y) \mid y\\
        &(z \mid x \land z \mid y) \to (z \mid \gcd(x,y)),
    \end{align*}
    and also a $\PV$-function symbol $\xgcd$ for which $\PV_1$ proves
    \begin{align*}
        (\xgcd(x,y) = (g,u,v)) \to (g=\gcd(x,y) \land g=ux+vy).
    \end{align*}
\end{lemma}

Throughout this work, we introduce $\PV$-symbols by declaring their name and properties inside the Lemma environment. Therefore, in the subsequent parts of this work, $\gcd$ and $\xgcd$ in any extension of $\PV_1$ denote specific $\PV$-symbols which satisfy the statement of the Lemma~\ref{lemmapveuclid}. We are now ready to define the first algebraic axiom we need for our formalization. Throughout the paper, we will use calligraphic uppercase letters $\X$, $\Y$, $\dots$ for formal variables of polynomials.

\begin{definition}[Generalized Fermat's Little Theorem]\label{definitiongflt}
Let $\GFLT$ be the natural $\Sigma^b_1$-sentence formalizing the following statement, where the exponentiation of polynomials is given by a $\PV$-symbol for exponentiation by repeatedly squaring modulo the polynomial $\X^r-1$.

\vspace{0.5em}
\fbox{\begin{minipage}{30em}
Let $p$ be a prime, and $a$ and $1^{(r)}$ be numbers such that $\gcd(a,p)=1$ and $r<p$. Then
\[
(\X+a)^p \equiv \X^p + a \md{p, \X^r-1}.
\]
\end{minipage}}
\end{definition}

The validity of $\GFLT$ follows in a sufficiently strong meta-theory by applying binomial theorem to the expresison $(\X+a)^p$ and considering which coefficients are divisible by $p$, this gives us $(\X+a)^p \equiv \X^p + a \md{p}$ which we take modulo $\X^r-1$. Since there are exponentially many monomials in the expansion of $(\X+a)^p$, it is not clear how to adapt the proof strategy in $\PV_1$ even if we assume the ordinary Fermat's Little Theorem as an axiom.

Let us now discuss the formalization of algebraic concepts in $\VTC^0_2$. As an extension of $T_2$, all function and relational symbols on the number sort are carried over from $\PV_1$. More interesting is the case of arithmetical operations over the set sort, where we interpret the set $X$ as a number $\sum_{u\in X}2^u$. Already $V^0_2$ can prove the totality and basic properties of addition and ordering. However, multiplication is $\TC^0$-complete. Thus, it is known that $V^0_2$ cannot prove it is total as this would contradict known lower bounds against $\AC^0$. Coding of sequences of elements of the set sort can be developed in $\VTC^0_2$, with $i$-th element of a sequence $S$ being $S^{[i]}$. It is also well-known that $\VTC^0_2$ can prove the totality and basic properties of multiplication. Moreover, $\VTC^0_2$ can $\Sigma^B_1$-define iterated addition $\sum_{i<n}X^{[i]}$ and prove
    \begin{align*}
        \sum_{i<0} X^{[i]} &= 0,\\
        \sum_{i<n+1} X^{[i+1]} &= X^{[n]}+\sum_{i<n} X^{[i]}.
    \end{align*}

Only recently has Jeřábek~\cite{jerabek2022iterated} shown that $\VTC^0_2$ can also $\Sigma^B_1$-define iterated multiplication and prove the analogous recursive properties. This also implies that it can define the division of integers.
\begin{theorem}[\cite{jerabek2022iterated}]\label{thrmimul}
    $\VTC^0$ can $\Sigma^B_1$-define iterated products $\prod_{i<n} X^{[i]}$ and prove the iterated multiplication axiom $\IMUL$:
    \begin{align*}
        \prod_{i<0} X^{[i]} &= 1,\\
        \prod_{i<n+1} X^{[i+1]} &= X^{[n]}\cdot\prod_{i<n} X^{[i]}.
    \end{align*}

    Moreover, $\VTC^0$ can $\Sigma^B_1$-define a function $\floor{X/Y}$ such that it proves \[Y\neq 0 \to \floor{X/Y}\cdot Y \leq X < (\floor{X/Y}+1)\cdot Y.\]
\end{theorem}

Finally, let us treat polynomials encoded by elements of the set sort as a sequence of coefficients, where each coefficient is an element in the number sort, which we shall call \emph{high degree polynomials}. The totality of addition of high-degree polynomials is straightforward, and using iterated addition, the totality of their multiplication can be proved in $\VTC^0$. We will observe in Theorem~\ref{thrmdlb} that the Kung-Sieveking $\TC^0$-algorithm as presented in \cite{healy2006fieldops} for the division of polynomials can also be proved correct in $\VTC^0_2$ which also implies the totality of division for high degree polynomials.

\section{The AKS algorithm and its correctness}

\begin{algorithm}
    \caption{The AKS algorithm~\cite{AKS}}\label{algaks}
 \hspace*{\algorithmicindent} \textbf{Input:} $n\geq 0$\\
 \hspace*{\algorithmicindent} \textbf{Output:} Whether $n$ is \texttt{COMPOSITE} or \texttt{PRIME}.
\begin{algorithmic}[1]
\State \text{Check whether $n=p^k$ for some $p\geq 0,k>1$, if yes output \texttt{COMPOSITE}.}
\State \text{Find the first $r$ such that $\ord_r(n) > \flr{\log n}^2$.}
\State \text{If some $a\leq r$ has $1<\gcd(a,n)<n$, then output \texttt{COMPOSITE}.}
\State \text{If $n\leq r$, output \texttt{PRIME}.}
\State \text{For $a = 0,\dots, \flr{\sqrt{\phi(r)}}\cdot \floor{\log n}$:} 
\State \qquad \text{If $(\X+a)^n \not \equiv \X^n + a$ mod $(\X^r-1,n)$: output \texttt{COMPOSITE}.}
\State \text{Output \texttt{PRIME}.}
\end{algorithmic}
\end{algorithm}
%(TODO: Could this be standalone section? Or is it too short? One could argue that this is no longer part of the preliminaries.) 
\subsection{Proof of the correctness}
We begin by providing a high-level overview of the proof of correctness of the AKS algorithm (Algorithm~\ref{algaks})~\cite{AKS}, here $\ord_r(n)$ is the multiplicative order of $n$ modulo $r$ and $\log$ is the base $2$ logarithm. In doing so, we restate and give new alphabetical names to key Lemmas and Theorems of~\cite{AKS} to avoid confusion with the numbering in our work. The key number-theoretic statement behind the algorithm is the following:

\begin{lemmm}[A] {\cite[Lemma 2.1]{AKS}} 
    Let $a\in \ZZ$, $n\in \NN$, $n\geq 2$ and $\gcd(a,n)=1$. Then $n$ is prime if and only if
    \[(\X+a)^n\equiv \X^n+a \md{n}.\]
\end{lemmm}

The AKS algorithm essentially emerges as an effectivization of this characterization of primality. As testing for this equality explicitly requires comparing objects of size exponential in $\abs{n}$, the AKS algorithm instead checks the equality modulo a low degree polynomial $\X^r-1$ for polynomially many values of $a$. Since the statement of this lemma cannot even be expressed in the language of $\PV$, in the formalization, we instead supplement it by the $\GFLT$ axiom. This axiom is what remains of Lemma A by keeping the only if direction and taking the equality mod $\X^r-1$ for $r$ polynomial in $\abs{n}$. 

This axiom implies the following direction of the correctness.

\begin{lemmm}[B] {\cite[Lemma 4.2]{AKS}}
    If $n$ is prime, the algorithm returns prime.
\end{lemmm} 

Next, the goal is to prove that the number $r$, which is essentially found by brute force in the AKS algorithm, always exists and has size polynomial in $\abs{n}$. This is Lemma D, which is, in turn, proved using Lemma C. Here, we use the symbol $\lcm$ to denote the function computing the least common multiple.

\begin{lemmm}[C] {\cite[Lemma 3.1]{AKS}} 
    Let $\LCM(m)$ denote the $\lcm$ of the first $m$ numbers. Then for $m\geq 7$: \[\LCM(m)\geq 2^m.\]
\end{lemmm}

\begin{lemmm}[D] {\cite[Lemma 4.3]{AKS}}
    There exists an $r \leq \max \{3, \lceil \log^5 n \rceil\}$ such that $\ord_r(n) > \floor{\log{n}}^2$.  %\log^2 n$.
\end{lemmm}

We use slightly different bounds in our formalization, but they are always polynomially related to the original ones.

From now on, we assume that on the input $n$ the AKS algorithm answered \texttt{PRIME}. We take $r$ from the statement of Lemma D and fix some prime $p \mid n$ such that $\ord_r(p) > 1$, the rest of the proof consists of showing that $n=p$.
Furthermore, we fix the number $\ell = \floor{\sqrt{\phi(r)}} \cdot \floor{\log n}$.

The proof then continues by defining the concept of \emph{introspectivity}.

\begin{definitionm}[{\cite[Definition 4.4]{AKS}}]
    For a polynomial $f(\X)$ and a number $m\in \NN$, we say that $m$ is \emph{introspective} for $f(\X)$ if
    \[[f(\X)]^m \equiv [f(\X^m)] \md{\X^r-1}.\]
\end{definitionm}

\begin{lemmm}[E] {\cite[Lemmas 4.5 and 4.6]{AKS}}
    If $m$ and $m'$ are introspective for $f(\X)$, then so is $m\cdot m'$. Moreover if $m$ is introspective for $f(\X)$ and $g(\X)$ then it is also introspective for $f(\X)\cdot g(\X)$
\end{lemmm}

Together Lemma A and our assumptions about $n$ and $p$ imply that $\frac{n}{p}$ and $p$ are both introspective for all $(\X+a)$, $0\leq a \leq \ell$. Lemma F then allows us to show that for the sets $I=\{(\frac{n}{p})^i p^j; i,j \geq 0\}$ and $P = \{\prod_{a}(\X+a)^{e_a};e_a\geq 0\}$ it holds that every number from $I$ is introspective for every polynomial in $P$.

Now we finally define $G = I \mdl r$ and $\G = P \mdl h$, where $h$ is an irreducible factor of the cyclotomic polynomial $Q_r$ over the field $\F_p$. After setting $t=\abs{G}$, we can state the last two technical Lemmas.

\begin{lemmm}[F] {\cite[Lemma 4.7]{AKS}}
    $\abs{\G} \geq \binom{t+\ell}{t-1}$
\end{lemmm}

\begin{lemmm}[G] {\cite[Lemma 4.8]{AKS}}
    If $n$ is not a power of $p$, then $\abs{\G} \leq n^{\floor{\sqrt{t}}}$.
\end{lemmm}

Since the AKS algorithm checks whether $n$ is a perfect power and we assume that it is not (otherwise, the algorithm outputs \texttt{COMPOSITE}), we get that both bounds on $\G$ hold. By showing they are incompatible we obtain:

\begin{lemmm}[H] {\cite[Lemma 4.9]{AKS}}
    If the algorithm returns \texttt{PRIME}, then $n$ is a prime.
\end{lemmm}

Putting Lemma B and Lemma H together, we get:

\begin{theoremm}[{\cite[Theorem 4.1]{AKS}}]
    The AKS algorithm outputs \texttt{PRIME} if and only if the number $n$ is a prime.
\end{theoremm}

\subsection{Overview of our formalization}

Let us start by the formalization of the AKS algorithm as a $\PV$-symbol. The algorithm begins by checking whether the number on the input is not a perfect power; this can be checked in polynomial time, and the algorithm as a $\PV$-symbol can be proved correct in $\PV_1$.

\begin{lemma}\label{lemmaperfectpower}
    There is a $\PV$-symbol $\PerfectPower$ for which $\PV_1$ proves:
    \[(\forall x)(\PerfectPower(x) \leftrightarrow ((x\leq 1)\lor (\exists a,y \leq x)( \abs{y} > 1 \land a^\abs{y} = x)))\]
\end{lemma}
\begin{proof}
    The $\PV$-symbol $\PerfectPower$ can be constructed to follow the procedure: For each value of $b\in\{1,\dots,\abs{x}+1\}$, use binary search to find a value of $a$ such that $a^{b}=x$ or continue if no such value exists. Such values for $a$ and $b$ will be found if and only if such values exist, which is equivalent to the existence of $a$ and $y=1^{(b)}$ such that $a^{\abs{y}}=x$.
\end{proof}

Next, one has to fix some upper bound for the number $r$, which is polynomial in $\abs{n}$. We will later prove in $S^1_2$ that such a bound exists. For convenience, let us also assume that the algorithm actually checks for the slightly stronger $\ord_r(n) > \abs{n}^2$. We show this stronger assumption is justified in Lemma~\ref{lemma D} by the same argument as in~\cite{AKS}. For any natural number $r$, the function $\phi(r)$
is \emph{Euler’s totient function} which outputs the number of numbers less than $r$ that are
relatively prime to $r$. The value $\phi(r)$ can be computed by a $\PV$-symbol which obtains $1^{(r)}$ on the input (see Lemma~\ref{lemmapvtot}). The function $\flr{\sqrt{x}}$ can be computed by a binary search for the value $y$ which satisfies $y^2 \leq x < (y+1)^2$. The function $\flr{\log x}$ is exactly equal to $\abs{x}-1$. All of these algorithms can be straightforwardly formalized as $\PV$-symbols. Altogether, the AKS algorithm can be formalized as a $\PV$-symbol $\AKSPrime$, which outputs $1$ if and only if the algorithm would output \texttt{PRIME}. By the correctness of the AKS algorithm, we mean the sentence $\AKSCorrect$:
\[(\forall x)(\AKSPrime(x) \leftrightarrow \Prime(x))\]

We prove the correctness by following the outline of the proof as in the previous section. For each Lemma or Theorem used in the formalization, we always strive to use the weakest possible theory for the given argument.

In Section~\ref{seccyclext}, we show that $S^1_2$ proves the existence of cyclotomic extensions of finite fields. This is important later in the argument to prove the existence of the polynomial $h$.

For the formalization, we will use the following congruence property. Recall that we fixed the number $r$ from Lemma~D satisfying $r \leq \max \{3, \lceil \log^5 n \rceil\}$ and $\ord_r(n) > \floor{\log{n}}^2$. We also fixed $\ell = \floor{\sqrt{\phi(r)}} \cdot \floor{\log n}$.
\begin{lemmm}[Congruence]
Let \(n\) be a number such that \(\AKSPrime(n)\) holds, i.e., the AKS algorithm
asserts that \(n\) is prime, and $r$ be the value provided by the algorithm satisfying \(\ord_r(n) > |n|^{2}\). Assume $p$ be a prime divisor of $n$ such that $\ord_r(p)>1$. And let us fix $\ell = \floor{\sqrt{\phi(r)}} \cdot \floor{\log n}$.
In $S^1_2+\GFLT$ it holds that for every $0 \leq a \leq \ell$
\[
(\X+a)^{\frac{n}{p}} \equiv \X^{\frac{n}{p}}+a \md{\X^r-1, p}.
\]
\end{lemmm}

We provide most of our formalization of the original proof in Section~\ref{secT2proof} in the theory $S^1_2 + \iWPHP + \DLB + \GFLT$. We start in Section~\ref{subseclegendre} by proving Legendre's formula in $\PV_1$; we use it in Section~\ref{subseclemmac} to prove our version of Lemma C, which is used in Section~\ref{subseclemmad} to prove Lemma D. In Section~\ref{subseccong}, we prove the Congruence Lemma and in Section~\ref{subseclemmae}, we formalize the concept of introspectivity and prove Lemma E. 

Since the property of a polynomial being in the group $\G$ is not easily formulated by a $\PV$-predicate, we instead use the predicate $\hat P$ recognizing all those polynomials over $F$ with degree at most $r$. Our version of Lemma F instead of an inequality states that there exists a function $\tau$ from $\{1,\dots, \binom{t+\ell}{t-1}\}$ into $\hat P$, 
which satisfies $x\neq y \to \tau(x) \nequiv \tau(y) \md {h, p}$. 
To prove Lemma F in Section~\ref{subseclemmaf}, we first formalize the Combinatorial Number System, which assigns to each number $i\leq \binom{m}{k}$ a unique $k$-element subset of $\{1,\dots,m\}$, when $m$ is given in unary. 
Our version of Lemma G, which again replaces inequality by an existence of a function $g:\hat P \to n^{\flr{\sqrt{t}}}$, 
with the property $f_1 \nequiv f_2 \md{h,p} \to g(f_1) \neq g(f_2)$, is proved in Section~\ref{subseclemmag} using the axiom $\DLB$. We finish in Section~\ref{subseclemmah} by proving Lemma H: the composition of the functions from Lemmas F and G results in an injective function which is in contradiction with the axiom $\iWPHP$. 

The remainder of the formalization is done in Section~\ref{secvtc}, where we prove that $\VTC^0_2$ proves the consequences of $S^1_2+\iWPHP+\DLB+\GFLT$. We thus obtain that both $\VTC^0_2$ and $T^\ct_2$ prove $\AKSCorrect$.
\section{Cyclotomic polynomials over finite fields in $\PV_1$}\label{seccyclext}

\subsection{Formalization of algebraic structures}

In this section, we shall prove in $\PV_1$ that over the finite fields $\ZZ/p$, where $p$ is prime, there are $r$-th cyclotomic polynomials for arbitrary $r<p$ represented in unary. We also define a concept of \emph{bounded fields} and use $S^1_2+\GFLT$ to prove that the $r$-th cyclotomic extension of $\ZZ/p$ exists. Bounded fields are simply finite fields on a domain $[b]=\{0,\dots, b-1\}$ for some number $b$, such that their operations are computed by boolean circuits. An analogous definition was already provided by Jeřábek in~\cite{jerabek2005weak}; here, we make the possibility of the field operation being coded by a circuit an explicit part of the definition to allow quantification over bounded fields.

\begin{definition}[$\PV_1$]
   A \emph{bounded field} is a tuple $(b,\tilde 0, \tilde 1,C_i,C_o,C_a,C_m)$ of elements, where $C_i$ is a circuit computing $i:[b]\to[b]$, $C_o$ is a circuit computing $o:[b]\setminus \{\tilde 0\} \to [b]\setminus \{\tilde 0\}$, $C_a$ is a circuit computing $a:[b]^2\to[b]$ and $C_m$ is a circuit computing $m:[b]^2\to[b]$ such that $a$ as addition, $m$ as multiplication, $i$ as additive inverse and $o$ as multiplicative inverse satisfy the field axioms on the universe $[b]$ with $\tilde 1$ as the $1$ element, and $\tilde 0$ as the $0$ element. That is, for all $x,y,z \in [b]$ we have
   \begin{align*}
       a(x,a(y,z))&= a(a(x,y),z) \tag{associativity of $a$}\\
       m(x,m(y,z))&= m(m(x,y),z) \tag{associativity of $m$}\\
       a(x,y)&=a(y,x)\tag{commutativity of $a$}\\
       m(y,x)&=m(y,x)\tag{commutativity of $m$}\\
       a(x,\tilde 0) &= x \tag{neutrality of $\tilde 0$}\\
       m(x, \tilde 1) &= x \tag{neutrality of $\tilde 1$}\\
       a(x, i(x)) &= \tilde 0 \tag{$i$ is the additive inverse}\\
       m(a(x,y),z) &= a(m(x,z),m(y,z)),\tag{distributivity}
   \end{align*}
   and for all nonzero $x\in [b]:$
   \[m(x,o(x))=\tilde 1. \tag{$o$ is the multiplicative inverse}\]
   
\end{definition}

When a bounded ring/field is understood from the context, we shall denote $a(x,y)$ as $x+y$, $m(x,y)$ as $x\cdot y$, $o(x)$ as $x^{-1}$, $i(x)$ as $-x$, $\tilde 1$ as $1$, and $\tilde 0$ as $0$.

\begin{lemma}[$\PV_1$]
    If $p$ is a prime, then $\ZZ/p=(p,1,0,C_i,C_o,C_a,C_m)$, where $C_i$, $C_o$, $C_a$, $C_m$ are the circuits computing the appropriate operations modulo $p$ is a bounded field.
\end{lemma}
\begin{proof}
    The existence multiplicative inverse follows from Lemma~\ref{lemmapveuclid} and the rest is immediate.
\end{proof}

In sections~\ref{seccyclext} and \ref{secT2proof}, the usual name for a bounded field is $F$. Low degree polynomials over a bounded field $F$ are defined as usual in $\PV_1$, that is as sequences of elements of $F$ of length equal to the degree of the polynomial plus one, and do not require special treatment. That is, for each bounded field $F$ there is a $\PV$-symbol which naturally encodes a predicate $f\in F[\X]$, we will therefore freely use this notation.

The following is straightforward, as the correctness of the Euclid's algorithm for polynomials can be proved analogously to the variant for integers.

\begin{lemma}[$\PV_1$, Euclid's algorithm]\label{lemmapveuclidpoly}
    Let $F$ be a bounded field, then there are $\PV$-symbols $\gcd$ and $\xgcd$ such that for any $f,g\in F[\X]$ we have
    \begin{align*}
        \gcd(f,g) &\in F[\X]\\
        \xgcd(f,g)&=(h,u,v)\\
        h,u,v&\in F[\X]\\
        h&=\gcd(f,g)\\
        h&=uf+vg,
        \end{align*}
 and
 \begin{align*}
        \gcd(f,g) & \mid f\\
         \gcd(f,g) & \mid g\\
        (\forall t\in F[\X])((t\mid f \land t & \mid g)  \to t \mid \gcd(f,g)).
    \end{align*}
    Moreover, $\gcd(f,g)$ is monic, that is the coefficient of the highest degree monomial is $1$.
\end{lemma}

\begin{lemma}[$\PV_1$]
    For a bounded field $F$, and $f\in F[\X]$ irreducible low degree polynomial, there is a bounded field $F[\X]/(f)$ such that the ring operations are interpreted as on polynomial operations modulo $f$.
\end{lemma}
\begin{proof}
    This is immediate after identifying polynomials of degree less than $k=\deg f$ with numbers below $b^k$, where $b$ is the bound of the field $F$. The existence of multiplicative inverse follows from Lemma~\ref{lemmapveuclidpoly}.
\end{proof}

\subsection{Existence of Cyclotomic extensions}

We will start by proving basic properties of the Euler's totient function $\phi$, which is definable in $\PV_1$ for numbers represented in unary.

\begin{lemma}[$\PV_1$]\label{lemmapvtot}
    There is a $\PV$-symbol $\Tot$, such that $\PV_1$-proves
    \[\Tot(1^{(r)})=\sum_{\substack{1\leq i \leq r\\ \gcd(i,r)=1}}1,\]
    we will simply denote the value $\Tot(1^{(r)})$ as $\phi(r)$.
\end{lemma}

\begin{lemma}[$\PV_1$]
    Let $1^{(r)}$ be a number, and $r$ be a prime, then $\phi(r)=r-1$.
\end{lemma}
\begin{proof}
    Since $r$ is a prime, then for $1\leq i < r$ we have $\gcd(i,r)=1$ and these elements simply form the set $\{1,\dots,r-1\}$ consisting of $r-1$ elements.
\end{proof}

\begin{definition}[$\PV_1$]
    Let $1^{(r)}$ be a number and let $d\leq r$. We define $S^r_d$ to be the number coding the set $\{0<m\leq r; \gcd(r,m)=d\}$.
\end{definition}

\begin{lemma}[$\PV_1$]\label{lemmacountsd}
    Let $1^{(r)}$ be a number and $d \mid r$, then $\abs{S^r_d} = \phi(r/d)$.
\end{lemma}
\begin{proof}
    We will prove the lemma by showing that division by $d$ defines a bijection between $S^r_d$ and the set $S^{r/d}_1 = \{0<m\leq r/d; \gcd(r/d,m)=1\}$.

    Indeed, assume $m \in S^r_d$, if $\gcd(m/d,r/d)=a>1$, then $a\cdot d$ is a common divisor of $m$ and $r$, a contradiction with $d$ being a greatest common divisor.

    On the other hand, assume $m\in S^{r/d}_1$, if $\gcd(r, m\cdot d)=a\cdot d$, we get that $a\cdot d$ divides both $(r/d)\cdot d$ and $m \cdot d$, then $a$ is a common divisor of $r/d$ and $m$ and thus $a=1$.
\end{proof}

\begin{lemma}[$\PV_1$]\label{lemmatotientsum}
    Let $1^{(r)}$ be a number, then $\sum_{d\mid r} \phi(d)=r$.
\end{lemma}
\begin{proof}
    We start with the observation that $\{1,\dots,r\}  = \bigcup_{d\mid r} S_d^r$. By taking the cardinality of both sides, we get 
    \begin{align*}
        r &= \sum_{d\mid r} \abs{S^r_d}\\
          &= \sum_{d \mid r} \phi(r/d)\tag{\dag}\\
          &= \sum_{d\mid r} \phi(d),\tag{\ddag}
    \end{align*}
    where ($\dag$) follows from Lemma~\ref{lemmacountsd} and ($\ddag$) follows from the fact that the map $d\mapsto r/d$ is an involution on the set of all divisors of $r$.
\end{proof}

We will now prove basic properties of polynomials over bounded fields which are needed to prove the existence of cyclotomic polynomials.

\begin{lemma}[$\PV_1$]\label{lemmaxmdiv}
    Let $1^{(k)}$ and $1^{(l)}$ be numbers, $k,l\geq 1$ such that $k\mid l$. Then, $\X^k-1\mid\X^l-1,$ over any bounded field.
\end{lemma}
\begin{proof}
    Assume that $l = km$. Then we have
    \begin{align*}
        (\X^k-1) (\sum_{i=1}^m \X^{(m-i)k}) &= (\sum_{i=1}^m \X^{(m-i)k+k}) - (\sum_{i=1}^m \X^{(m-i)k})\\
                                                          &= (\sum_{i=1}^m\X^{(m-i+1)k})-(\sum_{i=1}^m \X^{(m-i)k})\\
                                                          &= \X^l +(\sum_{i=2}^{m}\X^{(m-i+1)k})-(\sum_{i=1}^{m-1} \X^{(m-i)k})-1\\
                                                          &= \X^l +(\sum_{i=1}^{m-1}\X^{(m-i-1+1)k})-(\sum_{i=1}^{m-1} \X^{(m-i)k})-1\\
                                                          &= \X^l - 1. \qedhere
    \end{align*}
\end{proof}

\begin{lemma}[$\PV_1$]\label{lemmacoprimefactor}
    Let $F$ be a bounded field, $f,g,h\in F[\X]$ and $\gcd(g,h)=1$, then $\gcd(fg,h)=\gcd(f,h)$.
\end{lemma}
\begin{proof}
bounded    Assume that 
    \begin{align*}
        \xgcd(g,h) &= (1,u,v)\\
        \xgcd(f,h) &=(g_1,u_1,v_1)\\
        \xgcd(fg,h)&=(g_2,u_2,v_2),
    \end{align*}
    we will show that $g_1=g_2$. To prove that $g_2 \mid g_1$, we will multiply the equality $ug+vh=1$ by $g_1$ to obtain
    \begin{align*}
        (ug+vh)g_1&=g_1\\
        (ug+vh)(u_1f+v_1h)&=g_1\\
        ugu_1f + ugv_1h+vhu_1f+ vv_1h^2&=g_1\\
        fg(uu_1)+h(ugv_1+vu_1f+vv_1h)&=g_1,
    \end{align*}
    and since $g_2$ is a common divisor of both $fg$ and $h$, it is a divisor of $g_1$. Analogously, by multiplying the equality $ug+vh=1$ by $g_2$ we obtain
    \begin{align*}
        (ug+vh)g_2&=g_2\\
        (ug+vh)(u_2fg+v_2h)&=g_2\\
        ug^2u_2f + ugv_2h+vhu_2fg+ vv_2h^2&=g_2\\
    f(ug^2u_2)+h(ugv_2+vu_2fg+vv_2h)&=g_2,
    \end{align*}
    and $g_1$ is a common divisor of both $f$ and $h$, therefore a divisor of $g_2$.
\end{proof}

\begin{lemma}[$\PV_1$]\label{lemmaxmgcd}
    Let $1^{(k)}$ and $1^{(l)}$ be numbers, $k,l\geq 1$, then \[\gcd(\X^k-1,\X^l-1)=\X^{\gcd(k,l)}-1,\] over any bounded field.
\end{lemma}
\begin{proof}
    By induction on $\max\{k,l\}$. If $\max\{k,l\}=1$, the equality is trivial.

    Assume that the statement holds for all instances where $\max\{k,l\}<s$ for some $s$, where $1^{(s)}$ exist, we will prove it for the case $\max\{k,l\}=s$.  Without loss of generality, we can assume $k<l$, as the case $k=l$ is trivial. Note that this is an instance of the second mathematical induction for open formulas and it is available in $\PV_1$ because the number $\max\{k,l\}$ is a length as the corresponding quantifier in the induction formula is sharply bounded.

    Let $r,q$ be numbers such that $l=qk+r$, and $r<k$, by Lemma~\ref{lemmaxmdiv} there is a polynomial $f$, such that $f\cdot(\X^k-1)=(\X^{kq}-1)$. We have
    \begin{align*}
        \X^l-1 &= \X^{kq+r}-1\\ 
               &= \X^r(\X^{kq}-1)+\X^r-1\\    
               &= \X^r\cdot f \cdot (\X^k-1) + \X^r-1,
    \end{align*}
    and thus by Lemma~\ref{lemmacoprimefactor} \[\gcd(\X^k-1,\X^l-1)=\gcd(\X^k-1,\X^r-1)=\X^{\gcd(k,r)}-1=\X^{\gcd(k,l)}-1. \qedhere\]
\end{proof}

\begin{lemma}[$\PV_1$]\label{lemmacoprimeproduct}
    Let $1^{(k)}$ be a number for $k \geq 1$ and $F$ be a bounded field. Let $\langle f_1,\dots,f_k\rangle$ be a sequence of elements of $F[\X]$ and let $g\in F[\X]$ such that for all $i\in \{1,\dots,k\}$ we have $\gcd(g,f_i)=1$. Then $\gcd(g,\prod_{i=1}^k f_i) = 1$.
\end{lemma}
\begin{proof}
    By induction on $k$. The case where $k=1$ is trivial. Assume the lemma holds for $k-1$, then
    \begin{align*}
        \gcd(g,\prod_{i=1}^kf_i) &= \gcd(g,f_k \prod_{i=1}^{k-1}f_i)\\
        &=\gcd(g,\prod_{i=1}^{k-1}f_i)\tag{\dag}\\
        &=1\tag{\ddag},
    \end{align*}
    where ($\dag$) follows from the previous line by Lemma~\ref{lemmacoprimefactor} and ($\ddag$) follows from the induction hypothesis.
\end{proof}

\begin{lemma}[$\PV_1$]\label{lemmairreddivisors}
    Let $1^{(k)}$ be a number for $k \geq 1$ and $F$ be a bounded field. Let $\langle f_1,\dots,f_k \rangle$ be a sequence of elements of $F[\X]$, such that we have $\gcd(f_i,f_j)=1$ whenever $i\neq j$, and $g\in F[\X]$. If for every $1\leq i\leq k$ we have $f_i\mid g$, then $\prod_{i=1}^kf_i \mid g$.
\end{lemma}
\begin{proof}
    By induction on $k$. For $k=1$ the statement is trivial.

    Assume that the statement holds for $k-1$, and that we want to prove it for the sequence $\langle f_1,\dots,f_k\rangle$ and $g\in F[\X]$ satisfying for every $1\leq i \leq k: f_i\mid g$ and also that for $i\neq j$ we have $\gcd(f_i,f_j)=1$. Then for $h=\prod_{i=1}^{k-1} f_i$ we have $h\mid g$ and by Lemma~\ref{lemmacoprimeproduct} also $\gcd(h,f_k)=1$. Assume that $\xgcd(h,f_k) = (1,u,v)$, therefore $uh+vf_k=1$ and fix polynomials $s,t\in F[\X]$ satisfying $sh = g$ and $tf_k = g$. 

    Then we have 
    \begin{align*}
        uh+vf_k&=1\\
        uhg+vf_kg&=g\\
        uhtf_k + vf_ksh &= g\\
        hf_k(ut+vs)&=g,
    \end{align*}
    and thus $\prod_{i=1}^kf_k = hf_k \mid g$.
\end{proof}

\begin{definition}[$\PV_1$]
    Let $1^{(k)}$ be a number for $k \geq 1$ and $F$ be a bounded field. If $f\in F[\X]$ and $f=\sum_{i=0}^k a_i \X^i$ we define its \emph{derivative} $f'\in F[\X]$ to be given by $\sum_{i=0}^{k-1}(i+1)a_{i+1}\X^i$.
\end{definition}

\begin{lemma}[$\PV_1$]\label{lemmaderi}
    Let $F$ be a bounded field. If $f,g\in F[\X]$, then
    \begin{align*}
        (f+g)'&=f'+g'\\
        (fg)'&=f'g+fg'.
    \end{align*}
\end{lemma}
\begin{proof}
    Straightforward.
\end{proof}

\begin{lemma}[$\PV_1$]\label{lemmamultrootcrit}
    Let $F$ be a bounded field, $f \in F[\X]$, and $\gcd(f,f')=1$. If for $g \in F[\X]$: $g^2 \mid f$, then $\deg g = 0$.
\end{lemma}
\begin{proof}
    We proceed by proving the contrapositive. Assume $g^2\mid f$ and $\deg g $ is at least $1$. Then $f=hg^2$ and by Lemma~\ref{lemmaderi} we have $f' = h'g^2 + 2hg'g$. Therefore $g \mid f'$ and thus $g\mid \gcd(f,f')$. Hence, $\gcd(f,f')\neq 1$.
\end{proof}

We are now ready to prove the existence of cyclotomic polynomials in $\PV_1$, the proof we use differs from the standard proof by only using elementary concepts and thus not needing the definition of cyclotomic polynomials which uses the field of complex numbers.

\begin{theorem}[$\PV_1$]\label{cycltomexists}
    There is a $\PV$-symbol $\Cyc_p$ such that $\PV_1$ proves that for every prime $p$, and every number $1^{(r)}$, $1\leq r < p$, we have that $\Cyc_p(1^{(r)}) = Q_r$ is a degree $\phi(r)$ polynomial over $\ZZ/p$, such that
    \begin{align*}
        &Q_r \mid \X^r-1,\\
        \forall r'\in\{1,\dots,r-1\}:&\gcd (Q_r, X^{r'}-1)=1.
    \end{align*}
\end{theorem}
\begin{proof}
    Let us first define $\Cyc_p(1^{(r)})$ to be either $\X-1$ in the case that $r=1$ and otherwise $\X^r-1 / (\prod_{d\mid r, d<r} \Cyc_p(1^{(d)})%\Cyc(1^{(d)})_d
    )$, possibly discarding the remainder. In the rest of this proof, we will denote the value of $\Cyc_p(1^{(a)})$ by $Q_{a}$. We will later show that the product $\prod_{d\mid r, d<r} Q_d$ indeed divides $\X^r-1$.
    
   The proof continues by induction on the sharply bounded formula $\psi(r)$ which is formed as a conjuction of the following formulas:
    \begin{itemize}
        \item $(\forall r' \leq \abs{1^{(r)}})(r'\geq 1 \to (\X^{r'}-1 = \prod_{d\mid r'} Q_d))$
        \item $(\forall r' \leq \abs{1^{(r)}})(r'\geq 1 \to (\deg Q_{r'} = \phi(r')))$
        \item $(\forall r_1, r_2 \leq \abs{1^{(r)}})((r_1 \neq r_2 \land r_1 \geq 1 \land r_2 \geq 1) \to \gcd(Q_{r_1},Q_{r_2})=1).$
    \end{itemize}

    The formula $\psi(r)$ is valid when $r=1$, as $Q_1=\X-1$, and the third conjuct of $\psi(1)$ is vacuously true.

    Now for the induction step assume that $\psi(r-1)$ is true. We will use it to prove $\psi(r)$. For distinct divisors $d_1,d_2$ of $r$ which are not equal to $r$ we have that $\gcd(Q_{d_1},Q_{d_2})=1$ and by Lemma~\ref{lemmaxmdiv} every $d\mid r, d<r$ satisfies $Q_{d} \mid \X^{d}-1 \mid \X^r-1$, we can obtain by Lemma~\ref{lemmairreddivisors} that $\prod_{d\mid r, d<r} Q_d \mid \X^r-1$.

    Regarding the second conjuct, the induction hypothesis implies that
    \[\deg \prod_{d\mid r, d<r} Q_d = \sum_{d\mid r, d<r}\phi(r),\]
    by the definition of $Q_r$, we have that $\deg Q_r = r-\sum_{d\mid r, d<r}\phi(r)$ which is equal to $\phi(r)$ by Lemma~\ref{lemmatotientsum}.

    To finish the induction step it remains to show that $\gcd(Q_r, Q_{r'})=1$ for every $1 \leq r' < r$. Assume that $g = \gcd(Q_r,Q_{r'})$, then
        \[g \mid Q_r = \frac{\X^r-1}{\prod_{d\mid r, d<r} Q_d}\tag{\dag}\]
    Let $r'' = \gcd(r,r')$, by Lemma~\ref{lemmaxmgcd} and the properties of $\gcd$ we have
    \[g=\gcd(Q_r,Q_{r'}) \mid \gcd(\X^{r}-1,\X^{r'}-1) = \X^{\gcd(r,r')}-1=\prod_{d\mid r''}Q_d,\]
    thus $g \mid \prod_{d\mid r''} Q_d \mid \prod_{d\mid r, d<r} Q_d$, combining this with ($\dag$) gives $g^2 \mid \X^r-1$. This implies by Lemma~\ref{lemmamultrootcrit} that $\deg g =0$, as $r<p$ gives us $\gcd(\X^r-1,r\X^{r-1})=1$. This concludes the induction step.

    To obtain the statement of the theorem, it remains to prove the equality $\gcd(Q_r,\X^{r'}-1)=1$ for any $0<r'<r$. By the third conjunct of $\psi(r)$ and Lemma~\ref{lemmacoprimeproduct} we have
    \[\gcd(Q_r,\X^{r'}-1)=\gcd(Q_r,\prod_{d\mid r'} Q_d)=1.\qedhere\]
\end{proof}

\begin{lemma}[$S^1_2$]\label{lemmapolyfact}
    Let $1^{(k)}$ be a number for $k \geq 1$, $F$ be a bounded field and $f\in F[\X]$. Then there is a sequence $\langle h_1,\dots,h_k\rangle$ such that each $h_i$ is an non-constant irreducible member of $F[\X]$ and $\prod_{i=1}^k h_i = f$. 
\end{lemma}
\begin{proof}
    By $\sblma$ there is a longest sequence of non-constant polynomials $\langle h_1,\dots,h_n \rangle$ such that $\prod^{n}_{i=1}h_i = f$. The irreducibility of $h_i$'s follows from the maximality of the length of the sequence.
\end{proof}

To prove the existence of the polynomial $h$ from the proof of the correctness in~\cite{AKS}, we will first prove some basic properties of exponentiation by squaring.

\begin{lemma}[$\PV_1$, Recursive property of the exponentiation by squaring]\label{lemmaexpmultiplicative}
    Let $m \geq 1$ be a number, let $p$ be a prime and let $1^{(r)}$ be a number and let $f$ and $g$ be polynomials, then
    $f^m \equiv f^{m-1}\cdot  f \md{g,p}.$
\end{lemma}
\begin{proof}
We will proceed by open polynomial induction on~$m$. If $m=1$, the statement is trivial.

For the induction step, we will assume the statement for $\floor{m/2}$. The case where there is $l$ such that $m-1=2l$ is trivial, as by the definition of exponentiation by squaring we have
\[f^m \equiv (f^{l})^2 f \equiv f^{m-1} \cdot f \md{g,p}.\]
In the case where there is $l$ such that  $m=2l$, then we have by definition of exponentiation by squaring
\begin{align*}
    f^m&\equiv (f^l)^2 \md{g,p}\\
    f^{m-1} &\equiv (f^{l-1})^2 \cdot f\md{g,p},
\end{align*}
which implies
\begin{align*}
    f^{m-1}f &\equiv (f^{l-1})^2 \cdot f^2\\
                              &\equiv (f^{l-1} f)^2\\
                              &\equiv (f^l)^2 \tag{\dag}\\
                              &\equiv f^m \md{g,p},
\end{align*}
where $(\dag)$ follows from the induction hypothesis. The second congruence can be proved analogously.
\end{proof}

\begin{lemma}[$\PV_1$, Divisor lemma]\label{lemmaexpdivisor}
     Let $p$ be a prime, $l$ be a number and let $f$, $g$ and $g_0$ be polynomials. If $g_0\mid g$, then
     \[[f]_g^l  \equiv [f]_{g_0}^l \md{g_0,p},\]
     where $[f]_g^l$ and $[f]^l_{g_0}$ denote exponentiation by squaring modulo $g$ and $g_0$ respectively.
\end{lemma}
\begin{proof}
    By induction on $l$, if $l=0$ the statement is trivial.

    For the induction step, assume the statement holds for $l$. Then by Lemma~\ref{lemmaexpmultiplicative} we have
    \[[f]^{l+1}_g \equiv  [f]^{l}_g\cdot  f \md{g,p},\]
    since $g_0\mid g$, then also
    \begin{align*}
        [f]^{l+1}_g &\equiv  [f]^{l}_g\cdot  f \md{g_0,p}\\
                    &\equiv  [f]^l_{g_0} \cdot f \md{g_0,p}\tag{*}\\
                    &\equiv  [f]^{l+1}_{g_0} \md{g_0,p}\tag{\dag},
    \end{align*}
    where ($\dag$) follows from Lemma~\ref{lemmaexpmultiplicative} and (*) follows from the induction hypothesis.
\end{proof}

\begin{lemma}[$\PV_1$, Composition of congruences]\label{lemmaexpcompcong}
    Let $p$ be a prime and let $l$ and $k$ be numbers, and let $f_1$, $f_2$ and $g$ be polynomials. If
    $f_1^l \equiv f_2 \md{g,p}$ then,
    \[(f_1^l)^k \equiv f_2^k \md{g,p}.\]
\end{lemma}
\begin{proof}
    By induction on $k$. The case where $k=0$ is true trivially.

    Assume the statement holds for $k$. Then by Lemma~\ref{lemmaexpmultiplicative} 
    \begin{align*}
        (f_1^l)^{k+1} &\equiv (f_1^l)^{k} \cdot f_1^l\\
                    &\equiv f_2^{k} \cdot f_1^l\\
                    &\equiv f_2^{k} \cdot f_2 \equiv f_2^{k+1} \md{g,p}.\qedhere
    \end{align*}
\end{proof}

\begin{lemma}[$\PV_1$, Factor Theorem]\label{lemmafactortheorem}
    Let $F$ be a finite field, $\alpha \in F$ and $f\in F[\X]$. Then $f(\alpha)=0$ if and only if $\X-\alpha\mid f$.
\end{lemma}
\begin{proof}
    We will assume that $f(\alpha)=0$ and prove that $\X-\alpha\mid f$, the other direction is immediate. By long division, there are $q,r\in F[\X], \deg r <1$ such that
    $f= q\cdot (\X-\alpha)+r$. After evaluating both sides at $\alpha$, we get $0=f(\alpha)= q(\alpha)\cdot (\X-\alpha)+r(\alpha)=r(\alpha)$, since $\deg r=0$ it is a constant and therefore $r(\alpha)=r=0$. Thus, $f = q\cdot (\X-\alpha)$.
\end{proof}

\begin{corollary}[$S^1_2+\GFLT$]\label{Cor: h divisor}
    Let $p$ be a prime and $1^{(r)}$ be a number such that $1\leq r < p$ and $r \nmid p-1$. Then there is an irreducible divisor $h$ of $\X^r-1$ over $\ZZ/p$, which does not divide $\X^{r'}-1$ for any $0 < r' < r$ and $\deg h$ is at least $2$.
\end{corollary}
\begin{proof}
    By Theorem~\ref{cycltomexists}, there is a polynomial $Q_r$ which divides $\X^r-1$ while satisfying $\gcd(Q_r,\X^{r'}-1)=1$ for any $0<r'<r$. By Lemma~\ref{lemmapolyfact} there is an irreducible factorization of $Q_r=\prod_{i=1}^k h_i$. The $h=h_1$ satisfies that it is a divisor of $\X^r-1$ but not of any $\X^{r'}-1$, where $0<r'<r$. It remains to show that $h$ is of degree at least $2$.

    Assume for contradiction that $\deg h=1$, thus without loss of generality $h=\X-\alpha$, for some $\alpha \in \ZZ/p$. %$\alpha\in F$. 
    By the $\GFLT$ axiom, we have 
    \[(\X+\beta)^p \equiv  \X^p+\beta \md{\X^r-1,p},\]
    for every $\beta\in \{0,\dots,p-1\}$. Since $\X-1\mid \X^r-1$, Lemma~\ref{lemmaexpdivisor} gives us
    \[(\X+\beta)^p \equiv  \X^p+\beta \md{\X-1,p}.\tag{*}\]
    We also trivially have
    \begin{align*}
        \X &\equiv 1 \md{\X-1,p}\\
        \X+\beta &\equiv 1+\beta \md{\X-1,p},
    \end{align*}
    by Lemma~\ref{lemmaexpcompcong}, we have that
    \begin{align*}
        \X^p &\equiv 1^p \equiv 1 \md{\X-1,p}\\
        (\X+\beta)^p &\equiv (1+\beta)^p \md{\X-1,p},
    \end{align*}
    which together with (*) gives us
    \[(1+\beta)^p \equiv  1+\beta \md{\X-1,p}.\]
    Thus, by a linear substitution we have in $\ZZ/p$ that $\beta^p = \beta$, which implies $\beta^{p-1}=1$ for every $\beta\in\{0,\dots,p-1\}$. 

    Since $p-1=qr+m$, for some $q$ and $m$, where $0< m < r$, we have in $\mathbb{Z}/p$ that
    \[1 = \alpha^{p-1} = \alpha^{qr}\cdot \alpha^m = \alpha^m,\]
    and $\alpha^m \neq 1$ as this would imply by Lemma~\ref{lemmafactortheorem} that $h=\X-\alpha$ divides $\X^m-1$, which is a contradiction with the fact that $h$ does not divide $\X^{r'}-1$ for any $0<r<r'$.
\end{proof}

\subsection{The Root Upper Bound axiom}\label{subsecdlb}

In the proof of Lemma G in~\cite{AKS}, polynomials of degree polynomial in $n$ are considered, as such they are not covered by the formalism of low degree polynomials. Polynomials with unrestricted degree can be defined over $\PV_1$ in a limited way as follows. 

A \emph{sparse polynomial} is a list of pairs of monomials and coefficients, with the intended meaning being the polynomial consisting of the sum of those monomials with those coefficients. Note that the number of monomials in a sparse polynomial is bounded by the length of the list, therefore by a length of some number.  For sparse polynomials, we can define addition, multiplication and evaluation at an element of a bounded field in $\PV_1$. We will also use the notation $\deg f$ to denote the highest exponent of a monomial in the sparse polynomial $f$. In the rest of this section, we define the second algebraic axiom needed for our formalization which facilitates reasoning about sparse polynomials. It uses a new function symbol to provide the injection replacing the inequality between the number of roots of a polynomial over a field and the degree of the polynomial.

\begin{definition}[Root Upper Bound] Let $\PV(\iota)$ denote the language of $\PV$ extended by a new ternary function symbol $\iota$. The axiom $\DLB$ is the universal $\PV(\iota)$-sentence naturally formalizing the following statement.

\vspace{0.5em}
\fbox{\begin{minipage}{30em}
Let $F$ be a bounded field and $f\in F[\X]$ be a sparse polynomial, then
\[
\iota(F,f, x): \{\alpha \in F; f(\alpha) = 0\} \to [\deg f]
\]
is an injective function in the parameter $x$.
\end{minipage}}
\vspace{0.5em}

The injective function obtained from $\iota$ by fixing a bounded field $F$ and a sparse polynomial $f\in F[\X]$ will be denoted $\iota_{F,f}$.
\end{definition}

Since the axiom $\DLB$ introduces a new function symbol, we need to extend the axiom schemas to allow this new function symbol in the formulas that appear in them.

\begin{definition}
    The theory $S^1_2(\iota)+\DLB(\iota)$ is the $\PV(\iota)$-theory extending $S^1_2$ by the axiom $\DLB$ and allowing the function symbol $\iota$ in the $\sblm$ axiom scheme. The theory $S^1_2(\iota) + \iWPHP(\iota) + \DLB(\iota)$ is the $\PV(\iota)$-theory extending $S^1_2+\iWPHP$ by the axiom $\DLB$ and by allowing the function symbol $\iota$ to be used in the formulas of the axiom scheme $\sblm$ and extending $\iWPHP$ to allow $\PV(\iota)$-terms instead of just $\PV$-symbols. 

    For simplicity, we shall omit $\iota$ in the name of these theories and simply refer to them as $S^1_2 + \DLB$ and $S^1_2+\iWPHP+\DLB$. By the name $S^1_2+\iWPHP+\DLB+\GFLT$ we simply mean the theory extending  $S^1_2+\iWPHP+\DLB$ by the $\GFLT$ axiom.
\end{definition}
\section{The correctness in $S^1_2+\iWPHP+\DLB+\mathrm{GFLT}$}\label{secT2proof}

\subsection{$\PV_1\vdash$ Legendre's formula}\label{subseclegendre}

We start by proving Legendre's formula in $\PV_1$, this is subsequently used to prove our variant of Lemma C --- a lower bound on the value of $\lcm(1,\dots,2n)$ which is then used to prove our variant of Lemma D.

The following two Lemmas are straightforward and thus we omit their proofs.

\begin{lemma}
    There is a binary $\PV$-symbol $\nu_p(x)$ for which $\PV_1$ proves
    \begin{align*}
        \nu_p(0) &= 0\\
        \nu_p(x)&=\begin{cases}
            \nu_p(\floor{x/p})+ 1 & x \equiv 0 \md{p} \\
            0 & \text{otherwise.}
        \end{cases}
    \end{align*}
\end{lemma}

The value for $\nu_p(0)$ is often assigned to be $\infty$, but since we never use $0$ as an argument, we simply defined it to be $0$ to keep the function total.

The following lemma defines the bit-length factorial.
\begin{lemma}
    There is a $\PV$-symbol $\Fact{x}$ for which $\PV_1$ proves
    \begin{align*}
        \FactD{0} &= 1,\\
        \FactD{x} &= \FactD{\half{x}}\cdot \abs{x},
    \end{align*}
    that is:
    \[\NN \models \FactD{x} = \abs{x}!\]
\end{lemma}

\begin{lemma}[$\PV_1$]\label{lemmaprimeprime}
    Let $p$ be a prime and $x\cdot y \equiv 0 \md{p}$ then either $x\equiv 0 \md{p}$ or $y\equiv 0 \md{p}$.
\end{lemma}
\begin{proof}
    Assume that $p$ is a prime, $xy \nequiv 0 \md{p}$ and $x \nequiv 0 \md{p}$. Then by the fact that $p$ is prime we have $\gcd(x,p)=1$ and $\xgcd(x,p)=(1,u,v)$ such that $ux+vp=1$. Hence, $uxy+vpy = y$ and 
    \[
    y\equiv (uxy+vpy) 
        \equiv uxy 
        \equiv 0 \md{p}.   \qedhere
    \]
\end{proof}

\begin{lemma}[$\PV_1$]\label{lemmaprimepart}
    For every prime $p$ and a number $x$ there is $x'$ such that $x= x'\cdot p^{\nu_p(x)}$, and $p\nmid x'$.
\end{lemma}
\begin{proof}
    We will first use induction on the formula $\psi_1(x)$:
    \[i\leq \nu_p(x) \to p^i \mid x.\]

    The case where $i=0$ is trivial. Assume $\psi(i)$ and $i+1\leq \nu_p(x)$. By the inductive definition of $\nu_p(x)$, the case where $p^{i}\mid x$ but $p^{i+1}\nmid x$ is contradictory.

    We will now use induction on the formula $\psi_2(x)$:
    \[p^i \mid x \to \nu_p(x)=i+\nu_p(\floor{x/p^{i}}).\]
    Again, the case where $i=0$ is trivial. Assume $\psi(i)$ and $p^{i+1}\mid x$, then 
    \begin{align*}
        \nu_p(x)&=i+\nu_p(\floor{x/p^{i}})\tag{\dag}\\
                &=i+1+\nu_p(\floor{x/p^{i+1}}).\tag{\ddag}
    \end{align*}
    where ($\dag$) follows from the induction hypothesis, and ($\ddag$) follows from the definition of $\nu_p$.

    Together, we have established that\[(\forall i)(i\leq \nu_p(x) \leftrightarrow p^i \mid x).\tag{*}\] Let $k=\nu_p(x)$ and $x'=\floor{x/p^{k}}$. From $(*)$ we get that $p^k \mid x$, thus $x=x'p^k$. Finally, if $p\mid x'$, then $p^{k+1}\mid x'p^{k}=x$, which by $(*)$ implies that $k=k+1$, a contradiction.
\end{proof}

\begin{lemma}[$\PV_1$]\label{lemmaprimeortho}
    Let $p$ be a prime, let $x$ and $y$ be numbers and let $\nu_p(y)=0$. Then, $\nu_p(xy)=\nu_p(x)$.
\end{lemma}
\begin{proof}
    By Lemma~\ref{lemmaprimepart} there is $x'$ which is not divisible by $p$, such that $x=x'p^{k}$, where $k=\nu_p(p)$. We will proceed by induction on $i$ on the formula
    \[\nu_p(x'yp^i)=i.\]

    Assume that $i=0$. Since $p\nmid x'$ and $p\nmid y$, then by Lemma~\ref{lemmaprimeprime} $p\nmid x'y$, and thus $\nu_p(x'y) = 0$.

    Now assume that $\nu_p(x'yp^i)=i$, then \[\nu_p(x'yp^{i+1})=\nu_p(\floor{x'yp^{i+1}/p})+1 = \nu_p(x'yp^i)+1 = i+1,\]
    which completes the induction step. By putting $i=k$, we get that \[\nu_p(xy) =\nu_p(x'yp^{k})=k=\nu_p(x).\qedhere\]
\end{proof}

\begin{lemma}[$\PV_1$]\label{lemmanumult}
    For a prime $p$ and $x,y >0$ we have
    \[\nu_p(x\cdot y) = \nu_p(x) + \nu_p(y).\]
\end{lemma}
\begin{proof}
    By Lemma~\ref{lemmaprimepart} there are $x'$ and $y'$ not divisible by $p$ such that \[x=x'p^{\nu_p(x)},\quad y=y'p^{\nu_p(y)}.\]
    Then \[\nu_p(xy) = \nu_p(x'y'p^{\nu_p(x)+\nu_p(y)})=\nu_p(x)+\nu_p(y),\] where the last equality follows from Lemma~\ref{lemmaprimeortho}.
\end{proof}

\begin{lemma}[$\PV_1$]\label{lemmasplitsum}
    For every $1^{(n)}$ and a prime $p$ we have
    \[\nu_p(\Fact{1^{(n)}}) = \sum_{i=1}^n \nu_p(i).\]
\end{lemma}
\begin{proof}
    By induction on $n$. The case $n=0$ is obvious.
    
    Assume the statement holds for $n$, then by Lemma~\ref{lemmanumult}:
    \begin{align*}
        \nu_p(\Fact{1^{(n+1)}}) &= \nu_p(\Fact{1^{(n)}}\cdot (n+1))\\
        &= \nu_p(\Fact{1^{(n)}})+\nu_p(n+1)\\
        &= \sum_{i=1}^n \nu_p(i)+\nu_p(n+1)\\
        &= \sum_{i=1}^{n+1} \nu_p(i). \qedhere
    \end{align*}
\end{proof}

\begin{lemma}[$\PV_1$]\label{lemmanusum}
    Let $p$ be a prime and $x$ be a number, then 
    \[\nu_p(x) = \sum_{1\leq j\leq \abs{x},\:p^j\mid x} 1.\]
\end{lemma}
\begin{proof}
    By Lemma~\ref{lemmaprimepart} there is $x'$ such that $x=x'p^{\nu_p(x)}$. We will proceed by induction on the formula
    \[\nu_p(x'p^i) = \sum_{1\leq j \leq i;\:p^j \mid x} 1.\]

    The case where $i=0$ is trivial. Assume that the statement holds for $i$. Then,
    $\nu_p(x'p^{i+1}) = \nu_p(x'p^i)+1=i+1,$ which completes the induction step.
    The statement of the lemma is obtained by taking $i=\nu_p(x)$.
\end{proof}

\begin{lemma}[$\PV_1$]\label{lemmadivsum}
    Let $1^{(a)}$ and $b$ be numbers, and $b\neq 0$. Then
    \[\sum_{i=1,b\mid i}^a 1 = \floor{a/b}.\]
\end{lemma}
\begin{proof}
     By induction on $a$. The case where $a=0$ is trivial.

     Assume that 
     \[\sum_{i=1,b\mid i}^a 1 = \floor{a/b},\]
     if $b\mid a+1$, then 
     \[\sum_{i=1,b\mid i}^{a+1} 1 = 1+ \sum_{i=1,b\mid i}^{a} 1 = 1+\floor{a/b} =  \floor{(a+1)/b},\]
     and if $b\nmid a+1$, then
     \[\sum_{i=1,b\mid i}^{a+1} 1 = \sum_{i=1,b\mid i}^{a} 1 =\floor{a/b} =  \floor{(a+1)/b}. \qedhere\]
     
\end{proof}

We are now ready to prove Legendre's formula in $\PV_1$. Our proof is a direct adaptation of a standard proof.

\begin{theorem}[$\PV_1$, Legendre's formula]\label{thrmlegendre} 
    For every $1^{(n)}$ and a prime $p$ we have
    \[\nu_p(\Fact{1^{(n)}}) = \sum_{i=1}^\abs{n} \div{n}{p^i}.\]
\end{theorem}
\begin{proof}
    By direct computation:
    \begin{align*}
        \nu_p(\Fact{1^{(n)}})&=\sum_{j=1}^n \nu_p(j)\tag{A}\\
                             &=\sum_{j=1}^n \sum_{i,p^i\mid j}1\tag{B}\\
                             &= \sum_{i=1}^{\abs{n}}\sum_{1\leq j\leq n; p^i \mid j} 1\\
                             &= \sum_{i=1}^{\abs{n}} \floor{n/p^i},\tag{C}
    \end{align*}
    where (A) follows from Lemma~\ref{lemmasplitsum}, (B) follows from Lemma~\ref{lemmanusum} and (C) follows from Lemma~\ref{lemmadivsum}.
\end{proof}

\subsection{Lemma C}\label{subseclemmac}

\begin{lemma}[$S^1_2$]\label{lemmaprimedivisor}
    For every $n\geq 2$ there is a prime $p\leq n$ such that $p\mid n$.
\end{lemma}
\begin{proof}
    By $\sblm$ we have
    \[(\forall n)(\exists m\leq n )(\forall m' \leq n)((1<n \land \abs{m'} < \abs{m}) \to (1< m \land m \mid n \land m'  \nmid n)).\]
    
    Let $n>1$ be a number, then we obtain $m$ by the above formula. If $m$ is not a prime, then there is $m'\mid m$, $1<m'$ and $\abs{m'}<\abs{m}$, which is in contradiction with the formula.
\end{proof}

\begin{lemma}[$S^1_2$]\label{lemmavaldiv}
    For every $n$ and $m$: if every prime $p \leq n$ satisfies that $\nu_p(n)\leq \nu_p(m)$, then $n\mid m$.
\end{lemma}
\begin{proof}
    Assume that $n$ and $m$ satisfy that $\nu_p(n)\leq \nu_p(m)$ for every prime $p \leq n$. Let $n'=\floor{n/\gcd(m,n)}$ and $m' = \floor{m/\gcd(m,n)}$. If $n'=1$ then $n\mid m$, and we are done. We now prove that $n'>1$ is impossible.
    
    Assume for contradiction that $n' > 1$, therefore by Lemma~\ref{lemmaprimedivisor} there is a prime $p$ such that $p\mid n'$ and therefore $p \mid n$. If $\nu_p(m')>0$, then $\gcd(m,n)\cdot p$ is a common divisor of both $m$ and $n$, a contradiction with $\gcd(m,n)$ being the greatest common divisor. Therefore, we obtain that $\nu_p(m')=0$. 
    
    Let $k=\nu_p(\gcd(m,n))$. Then by Lemma~\ref{lemmanumult}, 
    \begin{align*}
        \nu_p(n) = \nu_p(n' \cdot \gcd(m,n)) \geq  k + 1
                                             >k + 0 = \nu_p(m'\cdot \gcd(m,n)) = \nu_p(m)
    \end{align*}
     a contradiction.
\end{proof}

%\begin{lemma}\label{lemmalcmdefinition}
%    There is a $\PV$-symbol $\LCM(x)$, such that $\PV_1$ proves that
%    \[(\forall y\leq \abs{x})(y \mid \LCM(x)),\]
%    and
%    \[(\forall y)((\forall z \leq \abs{x})( z \mid y) \to \LCM(x) \mid y),\]
%    we shall denote $\LCM(x)$ as $\lcm(1,\dots,\abs{x})$.
%\end{lemma}
%\begin{proof}
%    We can see that by putting
%    \[\LCM(x) = \begin{cases}
%        1 & \abs{x}=0\\
%        \floor{\frac{\LCM(\half{x})\cdot \abs{x}} {\gcd(\LCM(\half{x}),\abs{x})}} & \text{otherwise,}
%    \end{cases}\]
%    the required properties are satisfied.
%\end{proof}

%\begin{lemma}
    %    There is a $\PV$-symbol $\Binom(x,y)$ such that $\PV_1$ proves for each $1^{(n)}$ and $1^{(m)}$ that
    %    \[\Binom(1^{(n)},1^{(m)}) = 
    %    \begin{cases}
        %    \floor{\Fact{1^{(n)}}/(\Fact{1^{(m)}}\cdot \Fact{1^{(m-n)}})} & n\geq m \\
        %            0 & \text{otherwise,}\\
        %    \end{cases}\]
        %    we shall denote $\Binom(1^{(n)},1^{(m)})$ simply as $\binom{n}{m}$.
        %\end{lemma}
        %
        %\begin{lemma}[$\PV_1$]
            %    
            %\end{lemma}

\begin{lemma}[$S^1_2$]\label{lemmabinomdiv}
    For every $1^{(n)}$ we have
    that $(\Fact{1^{(n)}})^2$ divides $\Fact{1^{(2n)}}$.
\end{lemma}
\begin{proof}
    Let $p$ be a prime. By Lemma~\ref{lemmanumult} and Theorem~\ref{thrmlegendre} we have \[\nu_p(\Fact{1^{(n)}}^2)=2\cdot\sum_{i=1}^\abs{n} \floor{n/p^i}.\]
    Again by Theorem~\ref{thrmlegendre}
    \begin{align*}
        \nu_p(\Fact{1^{(2n)}})&=\sum_{i=1}^{\abs{2n}}\floor{2n/p^i}\\
        &\geq\sum_{i=1}^{\abs{2n}}2\cdot\floor{n/p^i}\\
        &\geq \nu_p(\Fact{1^{(n)}}^2),
    \end{align*}
    which by Lemma~\ref{lemmavaldiv} gives $\Fact{1^{(n)}}^2\mid \Fact{1^{(2n)}}$.
\end{proof}

The following Lemma defines order for multiplicative groups whose universe is bounded by a length. The restriction on the size makes all relevant argument easily formalizable in $\PV_1$ and thus we omit its proof.

\begin{lemma}\label{lemmalcmdefinition}
    There is a $\PV$-symbol $\lcm(\langle x_1,\dots,x_m\rangle )$, such that $\PV_1$ proves 
    %\[(\forall y \leq \sum_{i=1}^m x_m)((\exists i \leq m)(y \mid x_i) \to (y \mid \lcm(\langle x_1,\dots,x_m\rangle )),\]
    \[(\forall i\le m) \  (x_i \mid \lcm(\langle x_1,\dots,x_m \rangle))\]
    and
    \[(\forall y)((\forall i \leq m)( x_i \mid y) \to \lcm(\langle x_1,\dots,x_m \rangle ) \mid y).\]
\end{lemma}
\begin{proof}
    (Sketch) We can see that by putting
    \[\lcm(\langle x_1,\dots,x_m\rangle ) = \begin{cases}
        1 & m=0,\\
        0 & \exists 1\leq i \leq  m: x_i=0,\\
        \floor{\frac{\lcm(\langle x_1,\dots, x_{m-1}\rangle )\cdot x_m }{\gcd(\lcm(\langle x_1,\dots, x_{m-1}\rangle) ,x_m)}} & \text{otherwise,}
    \end{cases}\]
    the required properties are satisfied.
\end{proof}

            \begin{lemma}[$S^1_2$]\label{lemmabinomdivlcm}
                For every $1^{(n)}$, we have
                \[\floor{\Fact{1^{(2n)}}/\Fact{1^{(n)}}^2}\mid \lcm(1,\dots,2n).\]
            \end{lemma}
            \begin{proof}
                Let $p$ be a prime. By the length minimization principle for open $\PV$-formulas\footnote{This is available in $\PV_1$ by interating over the interval where we are checking for the least length.}, there is a minimal $r$ such that $2n < p^{r+1}$. Notice that $r<\abs{n}+1$ and that $p^r\leq 2n$.
                
                By Lemma \ref{lemmabinomdiv} we have
                \[\Fact{1^{(2n)}}= \floor{\Fact{1^{(2n)}}/\Fact{1^{(n)}}^2}\cdot\Fact{1^{(n)}}^2,\]
                thus by Lemma~\ref{lemmanumult}
                \[\nu_p(\floor{\Fact{1^{(2n)}}/\Fact{1^{(n)}}^2}) = \nu_p(\Fact{1^{(2n)}})- 2 \cdot \nu_p(\Fact{1^{(n)}}).\]
                
                Now by Theorem~\ref{thrmlegendre} we have
                \begin{align*}
                    \nu_p(\floor{\Fact{1^{(2n)}}/\Fact{1^{(n)}}^2}) &= \sum_{i=1}^{\abs{2n}} \floor{2n/p^i} - 2 \sum_{i=1}^{\abs{n}}\floor{n/p^i}\\
                    &= \sum_{i=1}^{\abs{n}+1} \floor{2n/p^i} - 2 \sum_{i=1}^{\abs{n}}\floor{n/p^i} \\
                    &= \sum_{i=1}^{r} \floor{2n/p^i} - 2 \sum_{i=1}^{r}\floor{n/p^i} \tag{*}\\
                    &= \sum_{i=1}^{r} (\floor{2n/p^i} - 2\floor{n/p^i}) \\
                    &\leq r,\tag{\dag}
                \end{align*}
                where (*) follows because $\floor{2n/p^i}=0$ for $i>r$ and $(\dag)$ follows from the fact that $\floor{2n/p^i}-2\floor{n/p^i}\leq 1$ for every $i\geq 0$. By the fact that $p^r\leq 2n$ we have \[\nu_p(\lcm(1,\dots,2n))\geq r\]
                and therefore by Lemma~\ref{lemmavaldiv} the statement follows.
            \end{proof}
            
            \begin{theorem}[$S^1_2$]\label{thrmlcmbound}
                For every $1^{(n)}$ we have
                \[2^n \leq \lcm(1,\dots,2n).\]
            \end{theorem}
            \begin{proof}
                By Lemma~\ref{lemmabinomdivlcm} it is enough to prove $2^n \leq \floor{\Fact{1^{(2n)}}/(\Fact{1^{(n)}}^2)}$. We proceed by induction on $n$. For $n=1$ the statement holds as 
                \[2^1 \leq 2!/(1!)^2= 2.\]
                
                Assume the inequality holds for $n$. Then
                \begin{align*}
                    \floor{\Fact{1^{(2(n+1))}}/\Fact{1^{(n+1)}}^2}2&=\floor{\frac{\Fact{1^{(2n)}}(2n+2)(2n+1)}{\Fact{1^{(n)}}^2(n+1)^2}}\\
                    &\geq\floor{\frac{\Fact{1^{(2n)}}(2n+2)}{\Fact{1^{(n)}}^2(n+1)}}\\
                    &=2 \cdot \floor{\Fact{1^{(2n)}}/(\Fact{1^{(n)}}^2)}\\
                    &\geq 2\cdot 2^n = 2^{n+1}. \qedhere
                \end{align*}
            \end{proof}
            
            \begin{corollary}[$S^1_2$, Lemma C]\label{corollarylcmbound}
                For every $1^{(m)}$ we have
                \[2^{\floor{m/2}} \leq \lcm(1,\dots, m).\]
            \end{corollary}
            \begin{proof}
                If $m=2n$ for some $n$, then this is Theorem~\ref{thrmlcmbound}. Otherwise, $m=2n+1$ for some $n$ and $2^{n} \leq \lcm(1,\dots,2n) \leq \lcm(1,\dots,2n+1)$.
            \end{proof}
            
            \subsection{Lemma D}\label{subseclemmad}

            \begin{lemma}
                There is a $\PV$-symbol $\ORD$ such that $\PV_1$ proves for every $y$ and $x$ satisfying $\gcd(x,\abs{y})=1$ that
                \[\ORD(y,x)=i \iff i>0 \land y^i \equiv 1 \md{\abs{x}} \land (\forall j < i)(y^j \nequiv 1 \md{\abs{x}}.\]
                We will write $\ord_r(y)$ to denote $\ORD(y,1^{(r)})$.
            \end{lemma}

            \begin{lemma}[$S^1_2$, Lemma D] \label{lemma D}
                For every $x\geq 2$, there is $r\leq 2\abs{x}^{6}$, such that $\ord_r(x)>\abs{x}^2$.
            \end{lemma}
            \begin{proof}
                Notice, that $\ord_r(x) = i$ implies  $r \mid (x^i-1)$. Let 
                \[a = x^{b}\prod_{i=1}^{\abs{x}^2}(x^i-1)< x \cdot x ^{\abs{x}^4} < 2^{\abs{x}^6},\]
                where $b=\abs{2\abs{x}^6}$.
                By Corollary~\ref{corollarylcmbound} we have
                \[2^{\abs{x}^6} \leq  \lcm(1,\dots, 2\cdot \abs{x}^6).\]
                Hence, $\lcm(1,\dots,2\cdot \abs{x}^6)$ does not divide $a$. By Lemma~\ref{lemmalcmdefinition} we have that there is $r\leq 2 \cdot \abs{x}^6$ such that $r$ does not divide $a$, and since this value is bounded by a length, we can take the smallest such $r$.
                
                We will show that $\gcd(r,x)=1$. First, notice that from $r\leq 2\abs{x}^6$ it follows that for any prime $p$ that $\nu_p(r)\leq b$. Also for any prime $p$ we have that $\nu_p(x)\geq 1$ implies $\nu_p(x^b)\geq b$.
                
                \textbf{Claim:} There is a prime $p$ such that $\nu_p(r) > \nu_p(a)$ and $\nu_p(x)=0$.
                
                \textit{Proof of claim.} First, if $\nu_p(r)\geq 1$ and $\nu_p(x)\geq 1$, then \[\nu_p(r)\leq b \leq  \nu_p(x^b) \leq \nu_p(a).\]

                Since $r\nmid a$, then by Lemma~\ref{lemmavaldiv} there is a prime $p$ such that $\nu_p(r)>\nu_p(a)$, which cannot happen if $\nu_p(x)\geq 1$. This proves the claim.

                Consider the value $r/\gcd(r,x)$, we have
                \[\nu_p(r/\gcd(r,x)) = \nu_p(r) >\nu_p(a),\]
                thus by Lemma~\ref{lemmavaldiv} we have that $r/\gcd(r,x)\nmid a$. Since $r$ was chosen as the smallest non-divisor of $a$, then $\gcd(r,x)=1$.
                
                We claim that $r$ has $\ord_r(x)>\abs{x}^2$. If $\ord_r(x)\leq \abs{x}^2$ then there is $i\leq \abs{x}^2$ such that \[r \mid (x^i-1) \mid \prod_{i=1}^{\abs{x}^2}(x^i-1),\]
                a contradiction.
            \end{proof}

            This lemma implies, that the \texttt{if} statement on line 4 of the AKS algorithm is relevant only for finitely many values. The case of correctness where this \texttt{if} statement results in the answer \texttt{COMPOSITE} can be ignored, because it can be expressed a true bounded sentence and thus proved in $\PV_1$.

\subsection{Congruence Lemma}\label{subseccong}
Let us denote by \(H_0(n,r)\) the following assumptions:
\(n\) is a number such that \(\AKSPrime(n)\) holds, i.e., the AKS algorithm
asserts that \(n\) is prime, and $r$ is the value provided by the algorithm satisfying \(\ord_r(n) > |n|^{2}\).

We show that there is a prime divisor of $n$ called $p$ such that $\ord_r(p)>1$.

\begin{lemma}[$S^1_2$]\label{lemmaintfact}
   Assuming $H_0(n,r)$,  there exists a sequence coding the prime factorization of $n$.
\end{lemma}
\begin{proof}
    By $\sblma$ there is a longest sequence of numbers $\langle p_1,\dots,p_m \rangle$ such that $\prod^{m}_{i=1}p_i = n$ and for all $i$ we have $p_i>1$. The primality of $p_i$'s follows from the maximality of the length of the sequence.
\end{proof}

\begin{lemma}[$S^1_2$]\label{lemmaprimedivn}
     Assuming $H_0(n,r)$,  there exists a prime divisor $p$ of $n$ such that $\ord_r(p)>1$.
\end{lemma}
\begin{proof}
    Assume that $\langle p_1,\dots, p_k \rangle $ is the prime factorization given by the previous lemma and assume for contradiction, that $\ord_r(p_i)=1$ for all $i\leq k$. Then $n= \prod_{i=1}^k p_i \equiv \prod_{i=1}^k 1 \equiv 1 \md{r}$, a contradiction.
\end{proof}

Let us fix $p$ to be a prime divisor of $n$ such that $\ord_r(p)>1$. The rest of Section~\ref{secT2proof} is dedicated to showing that $n$ satisfies $\Prime(n)$. To prove the congruence lemma we will use two simple properties of exponentiation by squaring.

\begin{lemma}[$\PV_1$, Evaluation lemma] \label{lem: evaluation}
    Let $p$ be a prime, $l$ be a number and let $f_1$, $f_2$ and  $g$ be low degree polynomials. Then, \[[f_1^l](f_2) \equiv ([f_1](f_2))^l \md{ [g](f_2),p},\] where $[g_1](g_2)$ denotes the composition of a polynomial $g_2$ with a polynomial~$g_1$.
\end{lemma}
\begin{proof}
    By induction on $l$. The case where $l=0$ is true trivially.

    Assume the statement holds for $l$. Then by Lemma~\ref{lemmaexpmultiplicative} 
    \begin{align*}
        [f_1^{l+1}](f_2) &\equiv [f_1^l](f_2) \cdot [f_1](f_2)\\
                    &\equiv ([f_1](f_2))^l \cdot [f_1](f_2)\\
                    &\equiv ([f_1](f_2))^{l+1} \md{[g](f_2),p}.\qedhere
        \end{align*}
\end{proof}

\begin{lemma}[$\PV_1$, Composition of exponentiation]\label{lem: composition}
    Let $p$ be a prime and let $l$ and $k$ be numbers, and let $f$ and $g$ be low degree polynomials. Then,
    \[((f)^l)^k \equiv f^{l\cdot k} \md{g,p}.\]
\end{lemma}
\begin{proof}
    By induction on $k$. The case where $k=0$ is true trivially.

    Assume the statement holds for $k$. Then by Lemma~\ref{lemmaexpmultiplicative} 
    \begin{align*}
        ((f)^l)^{k+1} &\equiv (f^l)^k \cdot f^l\\
                    &\equiv f^{l\cdot k} \cdot f^l\\
                    &\equiv f^{l\cdot k + l} \equiv f^{l\cdot (k+1)} \md{g,p},
    \end{align*}
    where the additivity of exponents also follows from Lemma~\ref{lemmaexpmultiplicative} by straightforward induction.
\end{proof}

\begin{lemma}[$S^1_2+\GFLT$, Congruence Lemma]\label{lem: p np n}
  Assuming $H_0(n,r)$ and  $a$ being a number such that $\gcd(a,p)=1$,  we have:
\[
(\X+a)^{\frac{n}{p}} \equiv \X^{\frac{n}{p}}+a \md{\X^r-1, p}.
\]
\end{lemma}

\begin{proof} 
By the notation $x \equiv y$, we mean $x\equiv y \md{\X^r-1, p}$ throughout this proof.
By $\AKSPrime(n)$ we have $(\X+a)^n \equiv \X^n+a \md{\X^r-1,n}$. We have
\begin{align}
    (\X+a)^n \equiv \X^n+a\\
    (\X+a)^p \equiv \X^p+a
\end{align}
The equivalence $(1)$ follows from the fact that $p$ is a prime divisor of $n$ and $(2)$ is an instance of the axiom GFLT. Since $r < p$ and $p$ is prime, there exist $u$ and $v$ such that $\xgcd(r,p)=(1,u,v)$. Thus, $pv \equiv 1 \md{r}$. Without loss of generality, we can assume that $v>0$, because otherwise, we can take $v' = v -rv>0$. Since $\X^v$ is a low degree polynomial (as $v <r$), by substituting $\X^v$ for $\X$ in $(1)$ using Lemma~\ref{lemmaexpdivisor} and Lemma \ref{lem: evaluation} we get
\[
(\X^v+a)^n \equiv \X^{v \cdot n}+a.
\]
To be concrete, by (1) we have:
\begin{equation}[\X+a]_{\X^{r}-1}^{n}\equiv[\X]_{\X^{r}-1}^{n}+a\md{\X^{r}-1,p}\end{equation}
By substituting $\X^v$ we get
\begin{equation}[\X+a]_{\X^{r}-1}^{n}(\X^{v})\equiv[\X]_{\X^{r}-1}^{n}(\X^{v})+a\md{\X^{vr}-1,p}\end{equation}
By Lemma~\ref{lem: evaluation}:
\begin{align}
[\X+a]_{\X^{vr}-1}^{n}(\X^{v})&\equiv[\X^{v}+a]_{\X^{vr}-1}^{n}\md{\X^{vr}-1,p}\\
[\X]_{\X^{vr}-1}^{n}(\X^{v})&\equiv[\X^{v}]_{\X^{vr}-1}^{n}\md{\X^{vr}-1,p}
\end{align}
By Lemma~\ref{lemmaexpdivisor}:
\begin{align}
[\X^{v}+a]_{\X^{vr}-1}^{n}&\equiv[\X^{v}+a]_{\X^{r}-1}^{n}\md{\X^{r}-1,p}\\
[\X+a]_{\X^{vr}-1}^{n}&\equiv[\X+a]_{\X^{r}-1}^{n}\md{\X^{r}-1,p}\\
[\X^{v}]_{\X^{vr}-1}^{n}&\equiv[\X^{v}]_{\X^{r}-1}^{n}\md{\X^{r}-1,p}\\
[\X]_{\X^{vr}-1}^{n}&\equiv[\X]_{\X^{r}-1}^{n}\md{\X^{r}-1,p}
\end{align}
Therefore
\begin{align*}
   [\X^{v}+a]_{\X^{r}-1}^{n} &\equiv [\X^{v}+a]_{\X^{vr}-1}^{n}\tag{\text{by (7)}}\\
                             &\equiv [\X+a]_{\X^{vr}-1}^{n}(\X^{v})\tag{by (5)}\\
                             &\equiv [\X+a]_{\X^{r}-1}^n(\X^{v})\tag{by (8)}\\
                             &\equiv [\X]_{\X^{r}-1}^{n}(\X^{v})+a \tag{by (4)}\\
                             &\equiv [\X]_{\X^{vr}-1}^{n}(\X^{v})+a\tag{by (10)}\\
                             &\equiv [\X^v]_{\X^{vr}-1}^{n}+a\tag{by (6)}\\
                             &\equiv [\X^v]_{\X^{r}-1}^{n}+a\tag{by (9)}\md{\X^r-1,p},
\end{align*}
or simply $(\X^v+a)^n \equiv \X^{v \cdot n}+a$.
By an analogous argument, we can use (2), Lemma~\ref{lemmaexpdivisor} and Lemma \ref{lem: evaluation} to get 
\[
(\X^v+a)^p \equiv \X^{v \cdot p} +a \equiv \X +a.
\]
By Lemma \ref{lem: composition}, we obtain
\[
((\X^v+a)^p)^{\frac{n}{p}} \equiv \X^{v \cdot n}+a. \tag{*}
\]
Moreover, using Lemmas \ref{lemmaexpcompcong} and \ref{lem: composition} we get:
\[
\X^{v \cdot n} \equiv \X^{v \cdot p \cdot \frac{n}{p}} \equiv (\X^{v \cdot p})^{\frac{n}{p}} \equiv \X^{\frac{n}{p}}.
\]
Therefore, by Lemma \ref{lemmaexpcompcong}, $(*)$ will become 
\[
(\X+a)^{\frac{n}{p}} \equiv \X^{\frac{n}{p}}+a.\qedhere
\]

\end{proof}

\subsection{Lemma E}\label{subseclemmae}
%(TODO: Redo after the introduction of the new Lemmas.) We will start by fixing two groups $G$ and $\mathcal{G}$. 
To prove Lemma E, let us first fix some notation. Let $p$ be a prime divisor of $n$ such that $\ord_r(p)>1$. For the ease of notation, we will denote $\ZZ/p$ by $F_p$. By Corollary \ref{Cor: h divisor}, for $F_p[\X]$ and $r$, there exists an irreducible divisor $h$ of $\X^r -1$ in $F_p[\X]$ such that $h \not\mid\X^{r'}-1$ for all $r' <r$ and $1< \deg(h) <r$. Take such $h$ and fix it and take $F= F_p[\X] / h(\X)$. The following lemma introduces a set $G_r$ represented by a $\PV$-predicate, which will be used in the rest of the paper.
%The set $\mathcal{G}$ is the set of all residues of polynomials in $P$ modulo $h(\X)$ and $p$, which is in fact a group generated by the elements $\X, \X+1, \ldots, \X+\ell$ in the field $F$. Let $t$ be the number of elements in $G$. Let $f(\X)$ and $g(\X)$ be elements of the field $F$ and  %$F= F_p[\X] / h(\X)$, where $h$ is the irreducible divisor of $\X^r-1$ over $F$ obtained by Corollary \ref{Cor: h divisor}. Let $m$ and $m'$ be elements in the group $G$.
%In the AKS paper, we need a notion of introspective numbers for polynomials. In this subsection, we only consider the definitions and lemmas for elements $f(\X), g(\X)$ of the field $F= F_p[\X] / h(\X)$, where $h$ is the irreducible divisor of $\X^r-1$ over $F$ obtained by Corollary \ref{Cor: h divisor}. We also consider the numbers $m, m'$ in the group $G$, which is the set of all residues of numbers in $I$ modulo $r$, where $I = \{(\frac{n}{p})^i \cdot p ^j \mid i,j \geq 0\}$. %In the following lemmas we use \cite[Lemma 3.3]{EmilAbelian}.

\begin{lemma}\label{Lem: symbol G}
  Assuming $H_0(n,r)$, there is a $\PV$-symbol $G_r(x)$ such that $\PV_1$ proves:
\[
G_r(x) \leftrightarrow (x <r) \wedge (\exists i, j\leq r) (x\equiv (n/p)^i \cdot p^j \md r).
\]
\end{lemma}
\begin{proof}
Suppose that the number $x$ is given. First, we check that $x <r$. Then, for each $i,j \leq r$, check whether $x \equiv (n/p)^i \cdot p^j \md{r}$. If such $i$ and $j$ exist, then $G_r(x)$ holds. This argument can be formalized in $\PV_1$ since by Lemma \ref{lemma D} we have $r \leq |n|^{10}$.
\end{proof}

We sometimes use the notation $m\in G_r$ to mean $G_r(m)$. By Lemma~\ref{lemma D}, the number $r$ is bounded by $\abs{n}^{10}$. The theory $\PV_1$ can assign a cardinality to sets bounded by a length (and $|m|^{10}$ is the length of $s(m)$ for a suitable term $s$). We denote the cardinality of the set $\{x; G_r(x)\}$ by $t$, where $t < r$.

\begin{lemma}[$\PV_1$] \label{lem: t<phi(r)}
  Assuming $H_0(n,r)$, we have $t=|G_r| \leq \phi(r).$
\end{lemma}
\begin{proof}
We already know
\[
\quad \gcd(n,r)= \gcd(p,r)=1 \qquad \text{and hence} \qquad \gcd(n/p,r)=1 \tag{*}
\]
We claim that for any $x$ such that $G_r(x)$ we have $\gcd(x,r)=1$. As $G_r(x)$, there exist $i,j \leq r$ such that $x\equiv (n/p)^i \cdot p^j \md r$. Suppose for the sake of contradiction that $\gcd(x, r)=c$ where $c \neq 1$. Thus,
\[
c \mid r \qquad \text{and} \qquad c \mid (n/p)^i \cdot p^j
\]
If $i \leq j$ then $ c \mid (n/p)^{i-j}$ which is a contradiction with $\gcd(n/p,r)=1$. If $i < j$, then $c \mid n^i p^{j-i}$, which is a contradiction with the left side of (*).
\end{proof}

\begin{lemma}[$S^1_2$]\label{Lem: intro}
Assume $H_0(n,r)$. Let $f$ be a sparse polynomial and $m$ a number. There is a ternary $\PV$ predicate $\int_{\X^r-1}(f,p,m)$ for which $S^1_2$ proves 
\[
\int_{X^r-1}(f, p, m) \leftrightarrow \big(f(\X)\big)^m \equiv f(\X^m)\md{\X^r-1, p}.\]
\end{lemma}

The number $m$ is called \emph{introspective} for $f(\X)$. %The following lemma indicates that the introspectivity of $m$ only depends on $m$ and $r$. The proof easily follows from Lemma \ref{lem: composition}. \begin{lemma}[$S^1_2$]\label{lem: introspective m independence}Let $f$ be a sparse polynomial and $m$ a number. Denote $m$ modulo $r$ by $m_r$. Then, $m$ is introspective for $f$ if and only if $m_r$ is introspective for $f$.\end{lemma}

Let us fix $\ell = \floor{\sqrt{\phi(r)}} \cdot \floor{\log n}$. We have the following easy observation. 

\begin{lemma}[$S^1_2+\GFLT$]\label{lem: p introspective}
  Assuming $H_0(n,r)$, both $p$ and $\frac{n}{p}$ are introspective for $\X + a$ for any $0 \leq a \leq \ell$.
\end{lemma}

\begin{proof}
By the axiom $\GFLT$ and Lemma \ref{lem: p np n}.
\end{proof}
The following lemma combines Lemmas 4.5 and 4.6 in \cite{AKS}. It proves the closure of introspective numbers under multiplication and also shows that the set of polynomials for which $m$ is introspective is closed under multiplication.

\begin{lemma}[$S^1_2$, Lemma E]\label{lem: intro mm'}
Assume $H_0(n,r)$. For any $f \in F[\X]$ and numbers $m$ and $m'$ if $m$ and $m'$ are introspective for $f(\X)$ then so is $m \cdot m'$. Moreover, for any $f,g \in F[\X]$ and number $m$ , if $m$ is introspective for $f(\X)$ and $g(\X)$ then it is also introspective for $f(\X) \cdot g(\X)$.
\end{lemma}

\begin{proof}
As $m$ is introspective for $f(\X)$ we have by Lemma~\ref{lemmaexpcompcong} and  Lemma~\ref{lem: composition}:
\[
\big(f(\X)\big)^{m \cdot m'} \equiv \big(f(\X^m)\big)^{m'}\md{\X^r-1, p} . \tag{\dag}
\]
Let $m = q\cdot r + m_r$, where $0\leq m_r<r$. Then, by Lemma~\ref{lemmaexpmultiplicative}, Lemma~\ref{lemmaexpcompcong}, Lemma~\ref{lem: composition}, and $\X^r \equiv 1 \md{\X^r-1,p}$, we have:
\[
\X^m \equiv \X^{q \cdot r +m_r} \equiv \X^{q\cdot r} \X^{m_r} \equiv (\X^r)^q \X^{m_r} \equiv 1^q \cdot \X^{m_r} \equiv \X^{m_r} \md{\X^r-1, p},
\]
which implies by Lemma~\ref{lem: composition}
\[\X^{m_r\cdot m'} \equiv (\X^{m_r})^{m'}\equiv (\X^{m})^{m'} \equiv \X^{m\cdot m'} \md{\X^r-1,p}.\]
By Lemma~\ref{lem: evaluation} we have
\begin{align*}
    [f(\X^{m_r})]_{\X^{m_r\cdot r}-1}^{m'}&\equiv [f]^{m'}_{\X^{m_r\cdot r}-1}(\X^{m_r}) \md{\X^{m_r\cdot r}-1}\tag{A}\\
    [\X^{m_r}]^{m'}_{\X^{m_r\cdot r}-1} &\equiv [\X]^{m'}_{\X^{m_r\cdot r}-1}(\X^{m_r}) \md{\X^{m_r \cdot r}-1}.\tag{B}
\end{align*}
As $\X^r-1 \mid \X^{m_r\cdot r}-1$, by Lemma~\ref{lemmaexpdivisor}:
\begin{align*}
    [f]_{\X^{m_r\cdot r}-1}^{m'} &\equiv [f]^{m'}_{\X^{r}-1}\md{\X^r-1,p}\\
    [\X]^{m'}_{\X^{m_r\cdot r}-1} &\equiv [\X]^{m'}_{\X^{r}-1} \md{\X^{r}-1,p}.
\end{align*}
Moreover, by general algebra
\begin{align*}
    [f]_{\X^{m_r\cdot r}-1}^{m'}(\X^{m_r}) &\equiv [f]^{m'}_{\X^{r}-1}(\X^{m_r})\md{\X^{m_r\cdot r}-1,p}\tag{C}\\
    [\X]^{m'}_{\X^{m_r\cdot r}-1}(\X^{m_r}) &\equiv [\X]^{m'}_{\X^{r}-1}(\X^{m_r}) \md{\X^{m_r\cdot r}-1,p}.\tag{D}
\end{align*}
By the introspectivity of $m'$
\[[f]^{m'}_{\X^r-1} \equiv f([\X]^{m'}_{\X^r-1}) \md{\X^r-1,p}.\]
Then,
\[[f]^{m'}_{\X^r-1}(\X^{m_r}) \equiv f([\X]^{m'}_{\X^r-1})(\X^{m_r}) \md{\X^{m_r\cdot r}-1,p}.\]
By (C) and (A)
\[ [f]^{m'}_{\X^{r}-1}(\X^{m_r}) \equiv [f]_{\X^{m_r\cdot r}-1}^{m'}(\X^{m_r})   \equiv [f(\X^{m_r})]_{\X^{m_r\cdot r}-1}^{m'} \md{\X^{m_r\cdot r}-1,p},\]
and by (D) and (B)
\begin{align*}
    f([\X]^{m'}_{\X^r-1})(\X^{m_r}) &\equiv f([\X]^{m'}_{\X^{m_r\cdot r}-1}(\X^{m_r}))\\
                                    &\equiv f([\X^{m_r}]^{m'}_{\X^{m_r\cdot r}-1}) \md{\X^{m_r\cdot r}-1,p}.    
\end{align*} 

Therefore, we have
\[{\big(f(\X^{m_r})\big)^{m'} } \equiv {f((\X^{m_r})^{m'}}) \md{\X^{m_r \cdot r}-1,p}.\]
Thus, by $\X^r-1 \mid \X^{m_r\cdot r}-1$ and Lemma~\ref{lemmaexpdivisor} we have
\[{\big(f(\X^{m_r})\big)^{m'} }  \equiv f(\X^{m_r\cdot m'})\md{\X^{r}-1,p}.\]
%and we know that
%\begin{align*}
%    f(\X^m) &\equiv f(\X^{m_r}) \md{\X^r-1,p }\\ 
%    f(\X^{m_r \cdot m'})&\equiv f(\X^{m\cdot m'})\md{\X^r-1,p}
%\end{align*}
%$${[f(\X^m)]^{m'} } \equiv f(\X^{m \cdot m'})\md{\X^{m \cdot r}-1, p}=f(\X^{m \cdot m'})\md{\X^r-1, p}.$$
By ($\dag$) we get $m \cdot m'$ is introspective for $f(\X)$:
\begin{align*}
\big(f(\X)\big)^{m \cdot m'} \equiv (f(\X^m))^{m'} &\equiv (f(\X^{m_r}))^{m'}\\
                                                &\equiv f(\X^{m_r\cdot m'}) \equiv f(\X^{m \cdot m'})\md{\X^r-1, p}.    
\end{align*}
For the second part of the lemma, by Lemma~\ref{lemmaexpmultiplicative} we have: 
$$
{\big(f(\X) \cdot g(\X)\big)^m } \equiv \big(f(\X)\big)^m \cdot \big(g(\X)\big)^m \equiv f(\X^m) \cdot g(\X^m)\md{\X^r-1, p}.\qedhere$$    
\end{proof}

\subsection{Lemma F}\label{subseclemmaf}
Recall that \(H_0(n,r)\) denotes the following assumptions:
\(n\) is a number such that \(\AKSPrime(n)\) holds, i.e., the AKS algorithm
asserts that \(n\) is prime, and $r$ is the value provided by the algorithm satisfying \(\ord_r(n) > |n|^{2}\).
From now on,  we let $H(n,r)$ denote the assumption $H_0(n,r)$ extended by fixing the following values:
\begin{itemize}
    \item 
    $\ell = \floor{\sqrt{\phi(r)}} \cdot \floor{\log n}$;
    \item 
    $t=|G_r|$;
    \item 
    $p$ as a prime divisor of $n$ such that $\ord_r(p)>1$;
    \item 
    $F_p=\ZZ/p$;
    \item 
    an irreducible divisor $h$ of $\X^r -1$ in $F_p[\X]$ such that $h \not\mid\X^{r'}-1$ for all $r' <r$ and $1< \deg(h) <r$;
    \item 
    $F= F_p[\X] / h(\X)$. 
\end{itemize}
Recall that $G_r$ is the set introduced in Lemma \ref{Lem: symbol G}. We start with an easy observation about $G_r$.
\begin{lemma}[$S^1_2+\GFLT$]\label{lem: distinct mm'}
Assume $H(n,r)$. For any distinct elements $m$ and $m'$ of $G_r$ we have 
\[
\X^m \nequiv \X^{m'} \md{h(\X), p}.
\]
\end{lemma}

\begin{proof}
We prove the lemma by contraposition. Suppose   
\[
\X^m \equiv \X^{m'} \md{h(\X), p}.
\]
Then, $h \mid \X^{m'}-\X^m$ and hence
\[
h \mid \X^{m} \cdot (\X^{m'-m}-1).
\]
However, since $m'<r$ and $m \neq m'$, by Corollary \ref{Cor: h divisor} we have 
\[
h \not \mid \X^{m'-m}-1.
\]
Therefore, $h \mid \X^m$. Since $h$ is an irreducible polynomial, this means that $h(\X)=\X$. However, this contradicts the fact that $\deg(h) >1$.
%which means that there exists an $s \leq m$ such that $h=\X^s$ and since $\deg(h) >1$ we have $s>1$. However, $h$ is an irreducible divisor of $\X^r-1$ which means that $\X^s \mid \X^r-1$, which is a contradiction. 
Thus, we must have $m=m'$.
\end{proof}

The following are a series of lemmas necessary for the proof of Lemma F. To state the next lemma, we use the  $\PV$-symbol $\NumOne(x)$ which counts the number of ones in the binary expansion of its argument.
\begin{lemma}[$\PV_1$]\label{Lem: ckm}
There is a $\PV$-symbol $c_{k}^{m}(x)$ such that $\PV_1$ proves for any numbers $1^{(k)}$ and $1^{(m)}$
\[
c_k^m(x) \leftrightarrow |x| \leq m \wedge \NumOne(x)=k.
\]
\end{lemma}

\begin{lemma}[$\PV_1$]\label{lem: injective function}
There is a $\PV$-symbol $f_{k}^m(i)$  such that $\PV_1$ proves for any numbers $1^{(k)}$ and $1^{(m)}$, where $0 \leq k \leq m$, 
\[
f_{k}^m: \left\{0<i \leq \binom{m}{k}\right\} \to \{x; c_k^m(x)=1\},
\]
is an injective function.
\end{lemma}
\begin{proof}
%(TODO: Rename $X^n_k$ to $c^m_k$ and $n$ to $m$, instead of `in the set $X^m_k$' have `satisfying $c^m_k$'. \textcolor{red}{Done.})
By recursion (simultaneously on $k$ and $m$) define the function $f_k^m$ as:
\[
f^0_0(1)=0 \quad \text{and} \quad
f_k^{m+1}(i) = \begin{cases}
    f_k^{m}(i) &  0 < i \leq \binom{m}{k} \\
     f_{k-1}^{m}(i-\binom{m}{k})+2^m & \binom{m}{k} < i \leq \binom{m+1}{k}
  \end{cases}
\]
Now, we prove the following claim.

\textbf{Claim:} If $0 < i \leq \binom{m+1}{k}$ then we have $c_k^{m+1}(f_k^{m+1}(i))$.
%\[\forall \; 0 < i \leq \binom{m+1}{k}: \quad c_k^{m+1}(f_k^{m+1}(i)).\]     
 
\noindent We prove the claim by induction on $k$ and $m$. For $k=m=0$ the claim trivially holds. Suppose by the induction step, for any $0 < i \leq \binom{m+1}{k}$, we have $c_k^{m}(f_k^m (i))$ and $c_{k-1}^{m}(f_{k-1}^m (i))$. Using the equality $\binom{m+1}{k}= \binom{m}{k}+ \binom{m}{k-1}$ and the definition of the predicate $c_k^{m+1}(x)$, it is easy to see that the following holds for any $x < 2^{m+1}$
\[
c_k^{m+1}(x)= c_k^{m}(x) \vee c_k^{m}(2^m+x) \tag{*}
\]
%$X_k^{n+1}$ as the concatenation of the following two ordered sets\[X_k^{n+1} = \{i < 2^n \mid \NumOne(i)=k\}\cup \{2^n+j \mid j <2^n, \NumOne(j)=k-1\}\]Denote $Y=\{a + 2^n \mid a \in X_{k-1}^n\}$. Hence, $X_k^{n+1} = {X_k^n} ^\frown Y$.Clearly, for any $x\in X_k^{n+1}$ we have $\NumOne(x)=k$. Moreover, either $x < 2^n$ or $2^n < x <2^{n+1}$. Therefore, the above set is indeed the set $X_k^{n+1}=\{i < 2^{n+1} \mid \NumOne(i)=k\}$. 
Now, if $0 < i \leq \binom{m+1}{k}$ then either $i \leq \binom{m}{k}$ or $\binom{m}{k} < i \leq \binom{m+1}{k}$. 
\begin{itemize}
    \item 
    If  $0 < i \leq \binom{m}{k}$, we have $f^{m+1}_k(i)=f^m_k(i)$ and by the induction step we have $c_k^m(f^m_k(i))$. By (*), we get $c_k^{m+1}(f^m_k(i))$. %Since $X_k^{n+1} = {X_k^n} ^\frown Y$, we have $f^{n+1}_k(i) \in X^{n+1}_k$. 
    \item 
    If $\binom{m}{k} < i \leq \binom{m+1}{k}$, we have $f_{k}^{m+1}(i)= f_{k-1}^{m}(i-\binom{m}{k})+2^m$. By the induction step, $c_{k-1}^{m}(f_{k-1}^m (i-\binom{m}{k}))$. Therefore, we get $c_{k}^{m+1}(f_{k-1}^m (i-\binom{m}{k})+2^m)$ and hence, by (*), we get $c_{k}^{m+1}(f_{k}^{m+1} (i))$.
\end{itemize}
This finishes the proof of the claim. Now, it is left to prove that $f_k^m$ is injective. By induction on $k$ and $m$ we prove that if $f_{k}^{m+1}(i)=f_{k}^{m+1}(j)$ then $i=j$. By the induction step, we know that
\[
\text{if} \; f_{k}^{m}(i)=f_{k}^{m}(j)\; \text{then} \; i=j \qquad \text{if} \; f_{k-1}^{m}(i)=f_{k-1}^{m}(j)\; \text{then} \; i=j
\]
There are three cases:
\begin{enumerate}
\item 
If $i,j \leq \binom{m}{k}$, then $f_{k}^{m+1}(i)=f_{k}^{m+1}(j)=f_{k}^{m}(i)=f_{k}^{m}(j)$. Therefore, by the induction step, we have $i=j$.
\item 
If $\binom{m}{k} < i,j \leq \binom{m+1}{k}$, then $f_{k}^{m+1}(i)=f_{k}^{m+1}(j)=f_{k-1}^{m}(i-\binom{m}{k})+2^m=f_{k-1}^{m}(j-\binom{m}{k})+2^m$. Again, by the induction step, we have $i=j$.
\item 
We show that the case where $i \leq \binom{m}{k}$ and $\binom{m}{k} <j \leq \binom{m+1}{k}$ is not possible. For the sake of contradiction, suppose otherwise. Then, 
\begin{align*}
f_{k}^{m+1}(i) & =f_{k}^{m+1}(j) & \text{(By the assumption)} \\
f_{k}^{m+1}(i) & =f_{k}^{m}(i) & \text{as} \; i \leq \binom{m}{k} \\
f_{k}^{m+1}(j) & =f_{k-1}^{m}(j-\binom{m}{k})+2^m & \text{as} \; \binom{m}{k} <j \leq \binom{m+1}{k}
\end{align*}
By the claim $c_k^{m}(f_k^m(i))$ holds, which by definition we get $|f_k^m(i)| \leq m$. We have, 
\[
|f_{k}^{m}(i)| = |f_{k-1}^{m}(j-\binom{m}{k})+2^m|.
\]
However, 
$|f_{k}^{m}(i)| \leq m$,
which is a contradiction with
\[
|f_{k-1}^{m}(j-\binom{m}{k})+2^m| > m.\qedhere
\]
\end{enumerate}
\end{proof}
\begin{definition}
A function $f: \{1, \ldots, m\} \to \{1, \ldots, m'\} $ is called \emph{strictly order-preserving (s.o.p.)} if for any $x , y \leq m$ it satisfies
\[
\text{if} \quad x<y \quad \text{then} \quad  f(x)<f(y).
\]
\end{definition}
Recall that $\ell = \floor{\sqrt{\phi(r)}} \cdot \floor{\log n}$ and $r$ is a length of a string. By Lemma \ref{lem: t<phi(r)}, we have $t=|G_r| \leq \phi(r)$, which means that both $t$ and $\ell$ are lengths. Lemmas \ref{lem: injective function}, \ref{lem: f to e}, \ref{lem: injective}, and \ref{lem: uv} are the ingredients needed to prove the original Lemma F in \cite{AKS}. In the former three lemmas, we provide injective functions, which we will compose to get the main result in Corollary \ref{cor: combination of lemmas}. 

\begin{lemma}[$\PV_1$]\label{lem: injective}
Assume $H(n,r)$. There is a $\PV$-symbol $h^{\ell}_{t}(x)$ such that $\PV_1$ proves that the function $h^{\ell}_{t}$ with the domain
\[
 \{x; c_{\ell+1}^{t+\ell}(x)=1\}
\]
and the codomain
\[
\{f \mid f \; \text{is a  s.o.p. function}, f: \{1, 2, \ldots, \ell+1\} \to \{1, 2, \ldots, t+\ell\}\}
\]
is an injective function.
\end{lemma}
\begin{proof}
The function $h^{\ell}_{t}(x)$ is defined by recursion, sending each $x$ to a function $f$ of the above form. The idea is as follows. Suppose $x$ is given. Let $x_1=x$. We start by defining $f(1)=j$ where $j$ is the smallest bit of $x_1$ that is one, i.e, for any $j' < j$ the $j'$th bit of $x_1$ is zero. This is possible in $\PV_1$ because of the following: there is a $\PV$-symbol enumerating small sets and going through it to find the smallest nonzero bit and output the index. To define $f(2)$, we first calculate $x_2=x_1-2^{j-1}$. Then, $f(2)=k$ where $k$ is the smallest bit of $x_2$ that is one. We continue till reaching $f(\ell+1)$. Note that in each step the number of ones in $x_{i+1}$ is one less than the number of ones in $x_i$. The function $f$ with the domain $\{1, \ldots, \ell+1\}$ is well-defined since there are $\ell+1$ ones in $x$. Moreover, $f(i) \in \{1, \ldots , t+\ell\}$ and it is a s.o.p. function. It is clear that the function $h^{\ell}_{t}$ defined this way is injective.
%To define the function $h$, take a set $S$ of the above form. If there is an order on the elements of $S$, then we can write it as $S_1, S_2, \cdots, S_{\ell+1}$. For each $1 \leq i \leq \ell+1$ we have $S_i \in \{1, \cdots , t+\ell\}$ and if $1 \leq i < j \leq \ell+1$ then $S_i < S_j$. Then define $h(i)=S_i$ for  $1 \leq i \leq \ell+1$. Clearly, $f(i) \in \{1, \cdots , t+\ell\}$ and it is a s.o.p. function. To see why $h$ is injective, suppose $h(S)=h(S')$. Then, there are s.o.p. functions $f$ and $f'$ such that $h(S)=f$ and $h(S')=f'$. This means that for $1 \leq i \leq \ell+1$ we have $f(i)=f'(i)$, i.e. $S_i=S'_i$ which means $S=S'$, thus $h$ is injective.
\end{proof}

\begin{lemma}[$\PV_1$]\label{lem: f to e}
Assume $H(n,r)$. There is a $\PV$-symbol $g^\ell_t(f)$ such that $\PV_1$ proves that the function $g^\ell_t$ with the domain
\[
\{f \mid f \;\text{is a  s.o.p. function}, f: \{1, 2, \ldots, \ell+1\} \to \{1, 2, \ldots, t+\ell\}\}
\]
and the codomain $\{(e_0, \ldots, e_\ell) ; \Sigma_{i=0}^\ell e_i \leq t-1\}$ is an injective function.
\end{lemma}

\begin{proof}
The function $g^\ell_t$ behaves as follows. It takes an $f$ of the above form and sends it to a tuple $(e_0, \ldots, e_\ell)$. Take subsets $s_i \subseteq \{1, \ldots, t+\ell\}$ for $0 \leq i \leq \ell$ as follows:
\begin{align*}
s_0&:= \{j \in \{1, \ldots, t+\ell\} ; j < f(1)\}\\
s_i&:= \{j \in \{1, \ldots, t+\ell \}; f(i) < j < f(i+1)\}\} \; \text{ for $1 \leq i \leq \ell$}
\end{align*}
Now, we take $g^\ell_t(f):= (|s_0|, \ldots, |s_\ell|)$, i.e., $e_i := |s_i|$ for each $0 \leq i \leq \ell$. We have to show that $\Sigma_{i=0}^\ell |s_i| \leq t-1$. Note that
\begin{align*}
\Sigma_{i=0}^\ell |s_i| &=|s_0|+\Sigma_{i=1}^\ell |s_i|  \\
& = f(1) -1 + \Sigma_{i=1}^\ell (f(i+1)-f(i)-1)\\
& = f(\ell+1)- (\ell +1)\\
& \leq (t+\ell) - (\ell +1) \tag{*}\\
& \leq t-1
\end{align*}
where the inequality (*) is derived by the fact that the codomain of $f$ is $\{1, \ldots, t+\ell\}$.
%Since $f$ is strictly order-preserving, for each $1 \leq i \leq \ell+1$ we have $i \leq f(i)$. Therefore, $\ell +1 \leq f(\ell +1) \leq t+\ell$, and we obtain $0 \leq f(\ell+1)- (\ell +1) \leq t-1$, i.e., $\Sigma_{i=0}^\ell e_i \leq t-1$.

We prove that $g^\ell_t$ is injective. Suppose $g^\ell_t(f)=g^\ell_t(f')$. Therefore, there exist sets $s_0,\dots,s_\ell$ and $s'_0,\dots,s'_\ell$ such that $(|s_0|, \ldots, |s_\ell|)=(|s'_0|, \ldots, |s'_\ell|)$. By induction on $1 \leq i \leq \ell+1$, we will show that $f(i)=f'(i)$. Since $|s_0|=|s'_0|$ we get
\[
|\{j \in \{1, \ldots, t+\ell ; j < f(1)\}|=|\{k \in \{1, \ldots, t+\ell ; k < f'(1)\}|,
\]
which means that $f(1)=f'(1)$. Now, suppose for an $i \leq \ell$ we have $f(i)=f'(i)$. Since $|s_i|=|s'_i|$ we have:
\small\[
|\{j \in \{1, \ldots, t+\ell \}; f(i) < j < f(i+1)\}|=|\{k \in \{1, \ldots, t+\ell\} ; f'(i) < k< f'(i+1)\}|.
\]
\normalsize Therefore, $f(i+1)=f'(i+1)$. 
\end{proof}

\begin{lemma}[$\PV_1$] \label{lem: uv}
Let $u$ and $v$ be sets with cardinality bounded by a length. Then we have:
\[
\text{if} \qquad u \neq v \qquad \text{then} \qquad \Pi_{b \in u} (\X +b) \neq \Pi_{b \in v} (\X +b)
\]   
\end{lemma}
\begin{proof}
For the sake of contradiction, suppose $\Pi_{b \in u} (\X +b) = \Pi_{b \in v} (\X +b)$ and $u \neq v$. Then, without loss of generality we can assume that there is an element $c \in u$ such that $c \notin v$. Therefore,
\[
\Pi_{b \in u} (\X+b) \equiv 0 \qquad \md{\X+c}
\]
and 
\begin{align*}
\Pi_{b \in u} (\X+b)  & \equiv \Pi_{b \in v} (\X+b)\\
& \equiv  \Pi_{b \in v} \big((\X+c) + (b-c)\big)\\
& \equiv  \Pi_{b \in v} (b-c)
\qquad \md{\X+c}    
\end{align*}
Therefore, 
\[
\Pi_{b \in v} (b-c) \equiv 0 \md{\X+c}.
\]
Hence, there exists $b \in v$ such that $b=c$, a contradiction with $c \notin v$.
\end{proof}

The following lemma introduces a set $\hat{P}_t$, which will be used later.

%(TODO: Consider degree upper bound $r$ vs $t$) \textcolor{red}{The bound $r$ definitely works (see the explanation below), but we need $t$ and $t \leq r$ (in fact by Lemma \ref{lem: t<phi(r)}, $t \leq \phi(r)$ because $t=|G|$ and $G$ is a subgroup of $Z_r^*$, so $|G| \leq \phi(r) \leq r$). The bound $r$ works, because $\hat{P}_t$ contains polynomials in $P$ mod $h$. So the degree of a polynomial in $\hat{P}_t$ is less than or equal to $\deg h$. Moreover,\[\deg h= \ord_r(p) > \log^2 p >1 \qquad \ord_r (p) \mid \phi(r) \leq r\] So, $1 < \deg h \leq r$. Do you think we should add this or an explanation like this somewhere? By Corollary \ref{Cor: h divisor} we have $1 < \deg h$.}

\begin{lemma}
  Assuming $H(n,r)$, there is a $\PV$-symbol $\hat P_t(f)$ such that $\PV_1$ proves:
    \begin{align*}
        \hat P_t(f) \leftrightarrow &\:\text{there is a sequence $\langle e_0,\dots, e_\ell \rangle$ such that:}\\
                  &\sum_{a=0}^\ell e_a < t,  f\equiv \Pi_{a=0}^\ell (\X +a)^{e_a} \md{h,p};\quad \deg f <\deg h.
    \end{align*}
\end{lemma}
\begin{proof}
Let $f$ be a given polynomial of degree $t$. The truth value of $\hat P_t(f)$ is computed as follows,
\begin{itemize}
\item 
if $f \nequiv 0 \md{\X+a}$ for all $0 \leq a \leq \ell$, then $\hat{P}_t(f)$ is false. 
\item 
Otherwise, take the smallest $a$ such that $f \equiv 0 \md{\X+a}$ and assign $f_1=\frac{f}{\X+a}$.
\end{itemize}
Repeat the step for $f_1$, i.e., 
\begin{itemize}
\item 
if $f_1 \nequiv 0 \md{\X+a}$ for all $0 \leq a \leq \ell$, then $\hat{P}_t(f)$ is false. 
\item 
Otherwise, take the smallest $a$ such that $f \equiv 0 \md{\X+a}$ and assign $f_2=\frac{f_1}{\X+a}$.
\end{itemize}
Repeat the step at most $t$ many times, where each step takes $\ell$ many divisions. If we reach some $i$ such that $f_i \equiv 1 \md{\X+a}$ for some $0 \leq a \leq \ell$, then $\hat{P}_t(f)$ is true. Otherwise, $\hat{P}_t(f)$ is false. To compute the sequence $\langle e_0,\dots, e_\ell \rangle$ from the process, initially assign $s_0=\langle e_0,\dots, e_\ell \rangle = \langle 0,\dots, 0 \rangle$. In step $i+1$
\begin{align*}
\text{if} \quad f_i=\frac{f_{i-1}}{\X+a} \quad \text{and} \quad s_i&=\langle e_0,\dots, e_a, \dots, e_\ell \rangle \quad \text{then}\\
s_{i+1}&=\langle e_0,\dots,  e_a+1, \dots, e_\ell \rangle
\end{align*}
Now, to finish the proof, we need the following easy facts:
\begin{enumerate}
\item 
For all $1 \leq a \leq \ell$ we have $\X+a \neq 0$ in $F$, because by Corollary \ref{Cor: h divisor} we have $\deg(h)>1$.
\item
The polynomials $\X, \X+1, \cdots, \X+\ell$ are all distinct in $F$. The reason is as follows. Let $1 \leq i \neq j \leq \ell$. Since
\[
\ell =\floor{\sqrt{\phi(r)}} \cdot \floor{\log n} < \floor{\sqrt{r}} \cdot \floor{\log n} <r \quad \text{and} \quad r <p,
\]
then $i \neq j$ in $F_p$.
\end{enumerate}
Moreover, the argument can be formalized in $\PV_1$ as $\ell= \floor{\sqrt{\varphi(r)}}\cdot \floor{\log n}$ and by Lemma \ref{lemma D} we have $r \leq |n|^{10}$. %and hence the polynomials used here have logarithmically sized degrees. 
%For each sequence $\langle e_0,\dots e_\ell \rangle$ such that $\sum_{a=0}^\ell e_a < r$, we check whether $f=\Pi_{a=0}^\ell (\X +a)^{e_a}$. This argument can be formalized in $\PV_1$ because the number of such sequences is equal to $\binom{r+\ell}{\ell+1} \leq \frac{(r+\ell)^{\ell+1}}{(\ell+1)!}$. Moreover, $\ell= \lfloor \sqrt{\phi(r)} \log n \rfloor$ and by Lemma \ref{lemma D} we have $r \leq |n|^{10}$.
\end{proof}

\begin{corollary}[$\PV_1$]\label{cor: combination of lemmas}
Assume $H(n,r)$. There is a $\PV$-symbol $\sigma_t^\ell(i)$ such that $\PV_1$ proves 
\[
\sigma^{\ell}_t: \binom{t+\ell}{\ell+1}\to  \hat{P}_t
\]
is an injective function.
\end{corollary}
\begin{proof}
Take the functions $f_{\ell+1}^{t+\ell}$, $g_t^\ell$, and $h^{\ell}_t$ as in Lemmas \ref{lem: injective function}, \ref{lem: f to e}, and \ref{lem: injective}, respectively. By composing these functions, we get $e_0, \ldots, e_\ell$ such that
\[
g^{\ell}_t(h^{\ell}_t(f_{\ell+1}^{t+\ell}(i)))=(e_0, \ldots, e_\ell) \quad \text{and} \quad \Sigma_{i=0}^\ell e_i \leq t-1.
\]
Define $\sigma^\ell_t(i)=\prod_{a=0}^{\ell} (\X+a)^{e_a}$. The function $\sigma^\ell_t$ is injective by Lemma \ref{lem: uv}  and by the fact that the functions $f_{\ell+1}^{t+\ell}$, $g^\ell_t$, and $h^\ell_t$ are all injective. 
\end{proof}

%The following lemma provides a lower bound on the number of elements in $\mathcal{G}$. 

\begin{lemma}[$\PV_1$, Lemma F]\label{Lemma 4.7}
  Assume $H(n,r)$. There is a $\PV$-symbol $\tau^\ell_t(x)$ such that $\PV_1$ proves that the function
\[
\tau^\ell_t: \binom{t+\ell}{t-1} \to \hat{P}_t
\]
satisfies the following condition:
\[
\text{if} \qquad x\neq y \qquad \text{then} \qquad \tau^\ell_t(x) \nequiv \tau^\ell_t(y) \md{h,p}.
\]
\end{lemma}

\begin{proof}
We know that $\binom{t+\ell}{\ell+1}=\binom{t+\ell}{t-1}$. So, we will use them interchangeably.
Using Corollary \ref{cor: combination of lemmas}, take the injective function 
\[
\sigma^\ell_t: \binom{t+\ell}{\ell+1}\to \{f \in \hat{P}_t ; \deg (f) < t\}.
\]
Consider the function 
\[
\tau^\ell_t(x) =  \mod{\sigma^\ell_t(x)}{h,p}.
\]
%\[ \sigma:  \{f \in \hat{P} ; \deg (f) < t\} \to \mathcal{G},\]where $\sigma$ sends any polynomial $f$ of degree less than $t$ to an element of $\mathcal{G}$ by computing $f$ modulo $p$ and $h(\X)$, which in the end will be in the field $F$. Thus, if $\sigma$ is injective, then $\sigma \circ \pi: \binom{t+\ell}{t-1} \to \mathcal{G}$ will be an injective function, and this finishes the proof by taking $\tau:=\sigma \circ \pi$. 
We have to show that 
\[
\text{if} \qquad x\neq y \qquad \text{then} \qquad \tau^\ell_t(x) \nequiv \tau^\ell_t(y) \md{h,p}.
\]
For the sake of contradiction, suppose $x \neq y$ but $\tau^\ell_t(x) \equiv \tau^\ell_t(y) \md{h,p}$. Denote $\tau^\ell_t(x)=f(\X)$ and $\tau^\ell_t(y)=g(\X)$. Let $m \in  G_r$, i.e., there are $0 \leq i,j \leq r$ such that
\[
m \equiv (\frac{n}{p})^i \cdot p^j \md{r}.
\]
Thus, 
%The rest of the proof is dedicated to showing that $\sigma$ is injective, i.e., \[\text{if} \qquad \sigma(f)=\sigma(g) \qquad \text{then} \qquad f=g.\]
%For the sake of contradiction, suppose $f \neq g$ but $\sigma(f)=\sigma(g)$, which means that $f(\X) \equiv g(\X) \md{h, p}$. Let $m \in G$. Thus, 
\[
\big(f(\X)\big)^m\equiv \big(g(\X)\big)^m \md{h, p}.
\]
By Lemmas \ref{lem: p introspective} and \ref{lem: intro mm'}, $m$ is introspective for both $f$ and $g$. Therefore, 
\[
\big(f(\X)\big)^m \equiv f(\X^m)\md{\X^r-1, p} \quad \big(g(\X)\big)^m \equiv g(\X^m)\md{\X^r-1, p}.
\]
%In other words, there exist $a_i$ and $b_i$ for $1 \leq i \leq 3$ such that \[\begin{cases}[f(\X)]^m-f(\X^m)= a_1 p + b_1 (\X^r-1)\\[g(\X)]^m-g(\X^m)= a_2 p + b_2 (\X^r-1)\\[f(\X)]^m-[g(\X)]^m= a_3 p + b_3 h(\X)\end{cases}\]
Since $h(\X)$ divides $\X^r-1$ we get 
\[
f(\X^m) \equiv g(\X^m) \md{h, p}.
\]
Define $Q(\Y)=f(\Y) - g(\Y)$. Then, $\X^m$ is a root of $Q(\Y)$ for each $m \in G_r$. %By the part \ref{2} of Lemma \ref{lem: rth root of unity}, the multiplicative order of $\X^m$ in $F$ is $r$. \textcolor{red}{(what is the previous sentence used for?!)}
%and we know that $G$ is a subgroup of $Z^*_r$. Note that $X$ is a primitive $r^{th}$ root of unity and 
%Note that the multiplicative order of $X$ in the field $F$ is $r$. (TODO: Next sentence should be reorganized?) 
Moreover, if $\X^m\equiv \X^{m'} \md{h, p}$ for each $m,m' \in G_r$, then by Lemma \ref{lem: distinct mm'} we get $m=m'$. Hence, there are $|G_r|=t$ many distinct roots of $Q(\Y)$ in the field $F$. However, the degree of $Q(\Y)$ is less than $t$ and we get a contradiction. Thus $\tau^\ell_t(x) \nequiv \tau^\ell_t(y) \md{h,p}$. This argument can be formalized in $\PV_1$ because the polynomials used here have logarithmically bounded degrees.%\todo{Here, we are considering integer polynomials: $f(Y)-g(Y)$.}
%Thus, there are at least $\ell+1$ distinct polynomials of degree one in $\mathcal{G}$. Hence, by Corollary \ref{cor: combination of lemmas} there are at least $\binom{t+\ell}{\ell+1}$ distinct polynomials of degree less than $t$ in $\mathcal{G}$.
\end{proof}

\subsection{Lemma G}\label{subseclemmag}
%\begin{lemma}\label{lem: lexi} There is a $\PV$-symbol $g(-,-)$ such that $\PV_1$ proves that  \[  g: \{ 0 \leq i,j \leq k\} \to \{1, 2, \cdots, (k+1)^2\} \] is an injective function. \end{lemma} \begin{proof} Define the function $g$ as \[ g(i,j):=i(k+1)+j+1. \] The function $g$ is well-defined since for $0 \leq i,j \leq k$ we have  \[ 1 \leq i(k+1)+j+1 \leq (k+1)^2. \] Now, we show that $g$ is injective, i.e., \[ \text{if} \qquad g(i,j)=g(i',j') \qquad \text{then} \qquad i=i' \quad \text{and} \quad j=j'. \] For the sake of contradiction, suppose $g(i,j)=g(i',j')$, but it is not the case that $i=i'$ and $j=j'$. Therefore, \[ g(i,j)=g(i',j')=i(k+1)+j+1=i'(k+1)+j'+1. \] It is easy to see that the cases ($i=i'$ and $j \neq j'$) and ($i\neq i'$ and $j = j'$) cannot happen. So suppose $i \neq i'$ and $j \neq j'$. W.l.o.g. let $i < i'$. We have \[ i(k+1)+j+1=i(k+1)+(i'-i)(k+1)+j'+1, \] which yields $j= (i'-i)(k+1)+j'$. This is not possible because $i'-i >0$ and $j \leq k$. Hence, $g$ is an injective function.  \end{proof}

\begin{lemma} [$\PV_1$] \label{lem: leading to G}
Let $x$,$y$ and $1^{(k)}$ be numbers such that for any numbers $a, b\leq k$  we have $x^a \neq y^b$. %$x$ is not a power of $y$. 
There is a $\PV$-symbol $f_{x,y}(i)$ such that $\PV_1$ proves  
\[
f_{x,y}:  \{1, 2, \ldots, (k+1)^2\} \to \{x^i y^j \mid 0 \leq i,j \leq k\}
\]
is an injective function.
\end{lemma}
\begin{proof}
For any $1 \leq l \leq (k+1)^2$ we can write $l=i(k+1)+j+1$ where $0 \leq i,j \leq k$. We can find $i$ and $j$ as follows and define the function $f_{x,y}$ as $f_{x,y}(l)=x^iy^j$ where
\[
i = \floor{\frac{l-1}{k+1}} \qquad \text{and} \qquad
j = l-1  \; \bmod{(k+1)}.
\]
It is clear that $0 \leq i, j \leq k$.
%\begin{align*} i &= l  \; \bmod{k+1} \tag{*}  \\  j &=l-i(k+1)-1 \tag{**} \end{align*} First, we show that for the function $f_{x,y}$ defined above, it holds that $0 \leq i,j \leq k$. By (*), we have $0 \leq i \leq k$. By (**) and $1 \leq l \leq (k+1)^2$, we get  \[1 \leq j+1+i(k+1) \leq (k+1)^2. \tag{\dag}\] Moreover, as $0 \leq i \leq k$ we have  \[1 \leq 1+i(k+1) \leq k(k+1)+1. \tag{\ddag} \]By (\dag) and (\ddag) we get $0 \leq j \leq k$. 
Now, we prove that $f_{x,y}$ is injective, i.e., 
\[
\text{if} \qquad f_{x,y}(l)=f_{x,y}(m) \qquad \text{then} \qquad l=m.
\]
Suppose $f_{x,y}(l)=f_{x,y}(m)$ for some $l,m \in \{1, \ldots, (k+1)^2\}$. Therefore, there are $0 \leq i,i',j',j' \leq k$ such that 
\begin{align*}
i &= \floor{\frac{l-1}{k+1}}  & i' &= \floor{\frac{m-1}{k+1}}\\ 
j &= l-1  \; \bmod{k+1}  & j' &=m-1 \; \bmod{k+1}
\end{align*}
and
\[
x^iy^j=x^{i'}y^{j'}.
\]
There are several cases:
\begin{itemize}
\item 
If $i > i'$ and $j < j'$, then $x^{i-i'}=y^{j'-j}$, which is a contradiction with the assumption that for any numbers $a, b$ we have $x^a \neq y^b$.
\item 
If $i<i'$ and $j>j'$, we have a situation similar to the previous case.
\item 
If $i > i'$ and $j > j'$, then $x^{i-i'} y^{j-j'}=1$. This only happens when $x=y=1$, which is again a contradiction with the assumption.
\item 
If $i<i'$ and $j<j'$, we have a situation similar to the previous case.
\end{itemize}
Thus, $i=i'$, $j=j'$, and $l=m$. Hence, $f_{x,y}$ is injective.
\end{proof}

\begin{lemma}[$\PV_1$]\label{lem: hat I}
Assume $H(n,r)$. There is a $\PV$-symbol $\hat{I}_t(y)$ for which $\PV_1$ proves 
\[
\hat{I}_t(y) \leftrightarrow \exists \; 0 \leq i,j \leq \lfloor \sqrt{t} \rfloor \quad y= \left(\frac{n}{p}\right)^i \cdot p^j.
\]    
\end{lemma}
\begin{proof}
Note that $t < r$ by the discussion after Lemma \ref{Lem: symbol G} and $r< |n|^{10}$ by Lemma \ref{lemma D}. Therefore, it is easy to see that given a number $y$, in $\PV_1$ we can check whether there are numbers $0 \leq i,j \leq \lfloor \sqrt{t} \rfloor$ such that $y= (\frac{n}{p})^i \cdot p^j$.
\end{proof}

\begin{lemma}[$S^1_2+\DLB$, Lemma G]\label{Lemma 4.8}
Assume $H(n,r)$. If $n$ is not a power of $p$ then there exists a function
\[
\hat{g}_t: \hat{P}_t \to \{1, \ldots, n^{\floor{\sqrt{t}}}\},
\]
which satisfies the condition:
\[
\text{if} \qquad f_1 \nequiv f_2 \md{h,p} \qquad \text{then} \qquad \hat{g}_t(f_1)\neq \hat{g}_t(f_2).
\]
\end{lemma}
\begin{proof}
%We can think of $\hat{I}$ as a subset of the set $I$. 

Assuming $n$ is not a power of $p$, then by Lemma  \ref{lemmaintfact}, there is a prime divisor $p'$ of $n$ distinct from $p$. This implies, that $(n/p)^a \neq p^b$ for any numbers $a,b\leq \floor{\sqrt{t}}$. Thus, we can use Lemma~\ref{lem: leading to G}, to obtain the function $f_{\frac{n}{p}, p}$
\[
f_{\frac{n}{p}, p}:  \{1, 2, \ldots, (\lfloor \sqrt{t} \rfloor+1)^2\} \to \{(\frac{n}{p})^i \cdot p ^j \mid 0 \leq i,j \leq \lfloor \sqrt{t} \rfloor\}
\]
or alternatively
\[
f_{\frac{n}{p}, p}:  \{1, 2, \ldots, (\lfloor \sqrt{t} \rfloor+1)^2\} \to  \{y \mid \hat{I}_t(y)=1\},
\]
which is an injective function. This means that the number of distinct numbers $y$ such that $\hat{I}_t(y)=1$ is at least $(\lfloor \sqrt{t} \rfloor +1)^2 >t$. %Recall that $G$ is the set of all residues of numbers in $I$ modulo $r$. 
As $|G_r|=t$, there are two numbers $m_1 > m_2$ such that 
\[
\hat{I}_t(m_1) , \hat{I}_t(m_2) \text{ and } m_1 \equiv m_2 \md{r}.
\]
It is easy to see that
\[
\X^{m_1} \equiv \X^{m_2} \md{\X^r-1}.
\]
%Take the following function $f: \hat{I} \to G$ such that $f(m)= m \;(\bmod \, r)$. We know that $|G|=t$. Hence, there exist $m_1 > m_2 \in \hat{I}$ such that $f(m_1)=f(m_2)$. Therefore, there exists a number $a$ such that $m_1 -m_2=ar$. Thus,\[X^r-1 \mid X^{m_2}(X^{ar}-1)=X^{m_1}-X^{m_2}\] and we have $X^{m_1}=X^{m_2}\; (\bmod X^r-1)$. 
We want to define the function $\hat{g}_t: \hat{P}_t \to n^{\lfloor \sqrt{t} \rfloor}$ such that 
\[
\text{if} \qquad \hat{P}_t(f_1), \; \hat{P}_t(f_2), \;  f_1 \nequiv f_2 \md{h,p} \qquad \text{then} \qquad \hat{g}_t(f_1)\neq \hat{g}_t(f_2).
\]
Let $f_1, f_2 \in \hat{P}_t$ such that $f_1 \nequiv f_2 \md{h,p}$. Thus, for $i \in \{1,2\}$ we get 
\begin{align*}
\big(f_i(\X)\big)^{m_1} & \equiv f_i(\X^{m_1})\\
& \equiv f_i(\X^{m_2}) \\
 & \equiv \big(f_i(\X)\big)^{m_2} \; \md{\X^r-1, p}
\end{align*}
Therefore, in the field $F$, we have $\big(f_i(\X)\big)^{m_1} = \big(f_i(\X)\big)^{m_2}$. Now, take the sparse polynomial $Q'(\Y)=\Y^{m_1} - \Y^{m_2}$. We have $\deg(Q')=m_1$ and $f_i(\X)$ is a root of $Q'(\Y)$ in the field $F$. 
%As $f(\X)$ was an arbitrary element of $\mathcal{G}$, the identity function with the domain $\mathcal{G}$ and the codomain $\{\text{roots of}\; Q'(\Y)\}$ is an injective function. Hence, $|\mathcal{G}| \leq |\{\text{roots of}\; Q'(\Y)\}|$. On the other hand, by 
By the axiom $\DLB$, the function 
\[
\iota_{F,Q'} : \{\Y \in F ; Q'(\Y) =0\} \to [\deg Q']
%\{\text{roots of}\; Q'(\Y)\} \to \{1, \cdots, \delta(Q')\}
\]
is injective. %, i.e., $|\{\text{roots of}\; Q'(\Y)\}| \leq m_1$. 
%Therefore, $|\mathcal{G}| \leq m_1$. 
Thus, the number of distinct roots of $Q'$ is at least $m_1$. In addition, we have $m_1 \leq (\frac{n}{p} \cdot p)^{\lfloor \sqrt{t} \rfloor} \leq n^{\lfloor \sqrt{t} \rfloor}$. Take the function $\hat{g}_t(\Y)$ as $\iota_{F,Q'}(\Y)$. By the above observations we get $\hat{g}_t(f_1) \neq \hat{g}_t(f_2)$ and this finishes the proof. 
\end{proof}

\subsection{Lemma H}\label{subseclemmah}
We start with some easy combinatorial facts.
\begin{lemma}[$\PV_1$]\label{lem: Lemma4.9, part1}
For any numbers $1^{(k)}$, $1^{(l)}$, and $1^{(s)}$ where $k \geq s$ we have $\binom{k+l}{k}\geq \binom{s+l}{s}$.    
\end{lemma}

\begin{proof}
Since $k \geq s$, we have $k+i \geq s+i$ for each $1 \leq i \leq l$. Therefore, $\Pi_{i=1}^{l} (k+i) \geq \Pi_{i=1}^{l} (s+i)$. Hence, $\frac{(k+l)!}{k!l!} \geq \frac{(s+l)!}{s!l!}$, which is  $\binom{k+l}{k}\geq \binom{s+l}{s}$.    
\end{proof}

\begin{lemma}[$\PV_1$]\label{lem: Lemma4.9, part2}
For any number $1^{(k)}$, where $k \geq 6$, we have $\binom{2k+1}{k} > 2^{k+2}$.    
\end{lemma}

\begin{proof}
We have  
\begin{align*}
 \binom{2k+1}{k} &= \frac{(2k+1)(2k) \ldots (k+2)}{k!}   \\
 & = (k+2) \Pi_{i=0}^{k-2}\frac{2(k-i)+(i+1)}{k-i}\\
 & > (k+2)2^{k-1} 
\end{align*}
If $k\geq 6$ then we have $(k+2)2^{k-1}  \geq 2^{k+2}$. Therefore, $\binom{2k+1}{k} > 2^{k+2}$. 
\end{proof}
%We state a property of Euler’s totient function which will be used later. \begin{lemma}\label{lem: Euler} \textcolor{red}{Maybe we do not need this lemma!}We have \[\ord_r(n) \mid \phi(r).\] \end{lemma} \begin{proof} As we know, $\phi(r)$ is the number of $0 \leq k \leq r$ such that $\gcd(r,k)=1$.   (TODO: Complete the proof.) \end{proof}

\begin{lemma}[$S^1_2 + \iWPHP + \DLB + \GFLT$, Lemma H]\label{Lemma 4.9}
Assume $H(n,r)$. If the algorithm returns $\mathrm{PRIME}$ then $n$ is prime.
\end{lemma}

\begin{proof}
By Lemma \ref{Lemma 4.7}, there is an injective function from $\binom{t+\ell}{t-1}$ to $\hat{P}_t$.
%Moreover, as $G_r$ is a group generated by $n$ and $p$ mod $r$ and
Recall the definition of $G_r$ in Lemma \ref{Lem: symbol G} and $|G_r|=t$. We have
\[
t  \geq  \ord_r(n)   > |n|^2 
\]
where the rightmost inequality holds by Lemma \ref{lemma D} and the leftmost inequality is clear by Lemma \ref{Lem: symbol G}. %holds because $G_r$ contains the group generated by $n$ mod $r$.  
Hence,
\begin{align*}
    t & \geq \floor{\sqrt{t}} \cdot \floor{\sqrt{t}} \\
    & \geq \floor{\sqrt{t}} \cdot |n| \tag{$t > |n|^2$}\\ 
    & > \floor{\sqrt{t}} \cdot \floor{\log n} \tag{$|n| =\lceil \log (n+1) \rceil > \floor{\log n}$}
\end{align*}
%\begin{align*}  t & > \lfloor \sqrt{t} \; |n| \rfloor \tag{$t^2 > t |n|^2$}\\    & \geq \lfloor \sqrt{t} \; \log n \rfloor \tag{$|n| > \log n$} \end{align*}
Consequently, $t-1 \geq \floor{\sqrt{t}} \cdot \floor{\log n}$. In Lemma \ref{lem: Lemma4.9, part1} substitute $t-1$ for $k$ and $\floor{\sqrt{t}} \cdot \floor{\log n}$ for $s$ and $\ell+1$ for $l$ to get 
\[
\binom{t+\ell}{t-1} \geq \binom{\ell+1 + \floor{\sqrt{t}} \cdot \floor{\log n}}{\floor{\sqrt{t}} \cdot \floor{\log n}}.
\]
%By Lemma \ref{lem: Euler} and the fact that $t=\ord_r(n)$ we have 
By Lemma \ref{lem: t<phi(r)}, we have $t \leq \phi(r)$. Hence, 
\[\ell= \floor{\sqrt{\phi(r)}} \cdot \floor{\log n} \geq \floor{\sqrt{t}} \cdot \floor{\log n}.\] Therefore,
\[
\binom{\ell+1 + \floor{\sqrt{t}} \cdot \floor{\log n}}{\floor{\sqrt{t}} \cdot \floor{\log n}} \geq \binom{2\floor{\sqrt{t}} \cdot \floor{\log n}+1}{\floor{\sqrt{t}} \cdot \floor{\log n}}.
\]
In Lemma \ref{lem: Lemma4.9, part2}, if we substitute $\floor{\sqrt{t}} \cdot \floor{\log n}$ for $k$, we get 
\[
\binom{2\floor{\sqrt{t}} \cdot \floor{\log n}+1}{\floor{\sqrt{t}} \cdot \floor{\log n}} > 2^{\floor{\sqrt{t}} \cdot \floor{\log n}+2}
\]
for $\floor{\sqrt{t}} \cdot \floor{\log n} \geq 6$, which we can assume to hold as the remaining cases form a true bounded sentence. Moreover, we have $2^{\floor{\sqrt{t}} \cdot \floor{\log n}+2} \geq 2 n^{\lfloor \sqrt{t} \rfloor}$. This means that there exists an injective function from $2 n^{\lfloor \sqrt{t} \rfloor}$ to $\hat{P}_t$. However, by Lemma \ref{Lemma 4.8}, we have if $n$ is not a power of $p$, then there exists an injective function from $\hat{P}_t$ to $n^{\lfloor \sqrt{t} \rfloor}$. This means that there exists an injective function 
\[
f: [2 n^{\lfloor \sqrt{t} \rfloor}] \to [n^{\lfloor \sqrt{t} \rfloor}]
\]
computed by a $\PV(\iota)$-term, which is a contradiction with the axiom scheme $\iWPHP$ (because what we do is essentially composing $\iota$ with $\PV$ functions and not
iterating $\iota$ in some p-time process). Thus, $n$ is a power of $p$. But, if $n=p^k$ for some $k>1$ then the algorithm would have returned $\mathrm{COMPOSITE}$ in Step 1. Hence $n=p$.
\end{proof}

\subsection{The correctness}

With the Lemma~\ref{Lemma 4.9} in hand we have everything we need to prove the correctness in $S^1_2+\iWPHP+\DLB+\GFLT$.

\begin{theorem}\label{theoremT2proof}
\[S^1_2+\iWPHP+ \DLB + \GFLT \vdash \AKSCorrect \]    
\end{theorem}
\begin{proof}
    The implication $\AKSPrime(x)\to \Prime(x)$ is proved in Lemma~\ref{Lemma 4.9}. The other implication, $\Prime(x) \to \AKSPrime(x)$, is simpler. The AKS algorithm simply checks some properties of the number $x$ and if they are satisfied, the number is proclaimed a prime, so this implication follows from the Lemma~\ref{lemmaperfectpower} and the $\GFLT$ axiom.
\end{proof}
\section{$VTC^0_2$ proves $S^1_2+\iWPHP+\DLB+\mathrm{GFLT}$}\label{secvtc}

In this section, we show that the axioms with which we extended $S^1_2$ to prove the correctness of the algorithm are actually provable in $VTC^0_2$. The theory $\VTC^0_2$ includes $T_2$ by definition. We then show that $\GFLT$ can be proved and then that $\DLB$ can be proved with $\iota$ being replaced by $\Sigma^B_1$-definable function. By the $\Sigma^B_0(\card)-\mathrm{COMP}$ scheme and the axioms defining counting sequences we obtain that $S^1_2(\iota) + \iWPHP(\iota) + \DLB(\iota)+\GFLT$ is provable in $\VTC^0_2$ and thus $\VTC^0_2 \vdash \AKSCorrect$.

%(TODO: make it shorter. Do we need to recall all three, i.e., $x, X, \X$?) As mentioned in Preliminaries, to distinguish between the elements of the number sort and the elements of the set sort, we will always use lowercase letters $x$, $y$, $\dots$ for the former, and uppercase letters $X$, $Y$, $\dots$ for the latter. We will keep using the calligraphic uppercase letters $\X$, $\Y$, $\dots$ for formal variables of polynomials. To be consistent with the earlier notation, there is one exception to this: The name of an arbitrary bounded field will still be denoted $F$ despite bounded fields still being objects of the number sort.

\subsection{$\Sigma^B_1$-definability of algebraic operations}

In this section, we will prove the factorial and binomial coefficient functions to be $\Sigma^B_1$-definable when the inputs are from the number sorts. Similarly, we will show that function computing the powers of high degree polynomials is $\Sigma^B_1$-definable, provided that the exponent is from the number sort. The following two lemmas follow straightforwardly from the provability of the axiom $\IMUL$ (see Theorem~\ref{thrmimul}).

\begin{lemma}[$\VTC^0$]
   There is a $\Sigma^B_1$-definition of a function $x \mapsto x!$, whose values are objects of the set sort, for which $\VTC^0_2$ can prove the inductive properties:
   \begin{align*}
       0!&=1\\
       (x+1)!&=(x+1)\cdot x!
   \end{align*}
\end{lemma}

\begin{lemma}[$\VTC^0$]
    There is a $\Sigma^B_1$-definition of a function $\binom{x}{y}$, whose values are objects of the set sort,  which for $y\leq x$ satisfies:
    \[\binom{x}{y}=\floor{\frac{x!}{y!(x-y)!}},\]
    and otherwise $\binom{x}{y}=0$.
\end{lemma}

The main principle behind the proof of the following Lemma is still the axiom $\IMUL$.

\begin{lemma}[$\VTC^0$]\label{lemmapolypow}
    There is a $\Sigma^B_1$-definable function, which to each high degree polynomial $P$ and each number $x$ assigns a high degree polynomial $P^x$ such that $\VTC^0_2$ proves for every $x$ and $P$:
    \begin{align*}
        P^0 &= 1 \\
        P^{x+1} & = P\cdot P^{x}.
    \end{align*}
\end{lemma}
\begin{proof}
   For $P= \sum_{i=0}^d A_i\X^i$ we define $P^x$ to be $\sum_{i=0}^{xd} B_i\X^i$ such that for any $0\leq i \leq xd$: $B_i=(\sum_{j_1+\dots+j_x = i}\prod_{k=1}^x A_{j_k}),$ where $0 \leq j_k \leq n$. This is $\Sigma^B_1$-definable as both iterated multiplication and iterated addition are.

   We shall prove the recursive properties of this function by induction. For $x=0$, the statement is clear. Assume the statement holds for $x$, then 
   \begin{align*}
       P\cdot P^{x} &= P \cdot \sum_{i=0}^{xd} \left( \sum_{j_1+\dots + j_x = i}\prod_{l=1}^x A_{j_l}\right) \X^i\\
                    &= \sum_{i=0}^{(x+1)d}\left(\sum_{j+k=i} A_i \cdot\left( \sum_{j_1+\dots + j_x = j}\prod_{l=1}^x A_{j_l}\right) \right)\X^i\\
                    &= \sum_{i=0}^{(x+1)d}\left( \sum_{j_1+\dots + j_x + k = i} A_k\prod_{l=1}^x A_{j_l}\right) \X^i\\
                    &= \sum_{i=0}^{(x+1)d}\left( \sum_{j_1+\dots + j_x + j_{x+1} = i} \prod_{l=1}^{x+1} A_{j_l}\right) \X^i\\
                    &= P^{x+1}. \qedhere
   \end{align*}
\end{proof}

\subsection{The binomial theorem and provability of $\GFLT$}

With the factorial and binomial coefficient functions $\Sigma^B_1$-defined, we will prove the binomial theorem, which then serves as a lemma to prove the $\GFLT$ axiom.

\begin{lemma}[$\VTC^0$, Pascal's triangle]\label{lemmapasc}
   Let $1 \leq x$ and $y\leq x$, then 
   \[\binom{x}{y} = \binom{x-1}{y}+\binom{x-1}{y-1}\]
   and
   \[\binom{x}{y}y!(x-y)! = x!\]
\end{lemma}
\begin{proof}
    By induction on $x$. If $x=1$, we can check that both equations hold. Assume the statement holds for $x$. Regarding the first equation, we have
    \begin{align*}
        \binom{x}{y}+\binom{x}{y-1}&=
    \floor{\frac{x!}{y!(x-y)!}}+
    \floor{\frac{x!}{(y-1)!(x-y+1)!}}\\
        &=
    \floor{\frac{x!(x-y+1)}{y!(x-y+1)!}}+
    \floor{\frac{x!\cdot y}{y!(x-y+1)!}}
    \end{align*}
    by the second equation in the induction hypothesis, both of the denominators divide the numerators and thus
    \begin{align*}
    \floor{\frac{x!(x-y+1)}{y!(x-y+1)!}}+
    \floor{\frac{x!\cdot y}{y!(x-y+1)!}} = \floor{\frac{x!\cdot (x+1)}{y!(x+1-y)!}} &= \binom{x+1}{y}.
    \end{align*}
    To prove the second equation, we start with
        \[\binom{x+1}{y}=\binom{x}{y}+\binom{x}{y-1}\]
      which we multiply by $y!(x+1-y)!$ to obtain
        \[\binom{x+1}{y}y!(x+1-y)!=\binom{x}{y}y!(x+1-y)!+\binom{x}{y-1}y!(x+1-y)!\]
    which by the application of the induction hypothesis becomes:
    \[\binom{x+1}{y}y!(x+1-y)!=x!(x+1-y)+x!y=x!(x+1)=(x+1)!\qedhere\]
\end{proof}

\begin{theorem}[$\VTC^0$]\label{thrmbinomial}
    For any numbers $a$ and $b$:
    \[(\mathcal{X}+a)^b=\sum_{i=0}^b \binom{b}{i} \mathcal{X}^i a^{b-i}\]
\end{theorem}
\begin{proof}
    By induction on $b$. If $b=0$, then both sides evaluate to $1$.

    Assume the statement holds for $b$. Then 
    \begin{align*}
        (\X+a)^{b+1} &= (\X+a)(\X+a)^{b}\\
                 &= (\X+a)\sum_{i=0}^b \binom{b}{i}\X^i a^{b-i}\\
                 &= \left(\sum_{i=0}^b \binom{b}{i}\X^{i+1} a^{b-i}\right) + \left(\sum_{i=0}^b \binom{b}{i}\X^{i} a^{b+1-i}\right)\\
                 &= \left(\sum_{i=1}^{b+1} \binom{b}{i-1}\X^{i} a^{b+1-i}\right) + \left(\sum_{i=0}^b \binom{b}{i}\X^{i} a^{b+1-i}\right)\\
                 &= \X^{b+1} + a^{b+1} + \sum_{i=1}^{b} \left(\binom{b}{i}+\binom{b}{i-1}\right )\X^{i} a^{b+1-i},
    \end{align*}
    which by Lemma~\ref{lemmapasc} equals
    \[\sum_{i=0}^{b+1} \binom{b+1}{i} \mathcal{X}^i a^{b+1-i}.\qedhere\]
\end{proof}

\begin{lemma}[$\VTC^0_2$]\label{lemmaeuclidsequence}
    Let $p$ be a prime, and let $X$ code a sequence of $m$ elements of the number sort. If for each $i\leq m:p\nmid X^{[i]}$, then $p \nmid \prod_{i=1}^m X^{[i]}$.
\end{lemma}
\begin{proof}
    By induction on $m$. For $m\leq 1$ this is clear. Assume the statement holds for $m$ and let $\tilde X=\prod_{i=1}^m X^{[i]}$ and $x=X^{(m+1)}$ and assume for contradiction that $p\mid \tilde X x$. By Lemma~\ref{lemmapveuclid} there is $\xgcd(x,p)=(1,u,v)$ such that $ux+vp=1$. Multiplying by $\tilde X$ we get
    \[ux\tilde X + v \tilde X p = \tilde X.\]
    As both summands on the left hand side are divisible by $p$, so is $\tilde X$, which is a contradiction.
\end{proof}

\begin{lemma}[$\VTC^0_2$]\label{lemmabinomdivisibility}
Let $p$ be a prime, then for every $0<m<p$ we have $p \mid \binom{p}{m}$.
\end{lemma}
\begin{proof}
    All numbers less than $p$ are not divisible by $p$, thus by Lemma~\ref{lemmaeuclidsequence} we have that $p\nmid m!(p-m)!$, and thus $m!(p-m)! \nequiv 0 \md{p}$.
    On the other hand by Lemma~\ref{lemmapasc} we have that $m!(p-m)! \mid m!$, therefore the following are equivalent
    \begin{align*}
        \binom{p}{m} &\equiv 0 \md{p}\\
        \binom{p}{m}m!(p-m)! &\equiv 0 \md{p}.
    \end{align*}
    By the second part of Lemma~\ref{lemmapasc} we have that  \[\binom{p}{m}m!(p-m)! = p! \equiv 0 \md{p}.\qedhere\]
\end{proof}

\begin{lemma}[The theory $\VTC^0_2$ proves $\GFLT$.]\label{lemmagflt}
    Let $p$ be a prime, $a\leq p$ a number and $1^{(r)}$ be a number such that $r<p$. Then
    \[(\X+a)^p \equiv \X^p + a \md{p, \X^r-1},\]
    where the exponentiation is computed using a $\PV$-symbol for exponentiation by squaring modulo $\X^r-1$.
\end{lemma}
\begin{proof}
    By Theorem~\ref{thrmbinomial} we have 
        \[(\mathcal{X}+a)^p=\sum_{i=0}^p \binom{p}{i} \mathcal{X}^i a^{p-i},\]
        by Lemma~\ref{lemmabinomdivisibility} we get
        \[(\mathcal{X}+a)^p\equiv \X^p + a \md{p}.\]

        To obtain the final congruence, we take both sides modulo $\X^r-1$. By Lemma~\ref{lemmaexpmultiplicative} and Lemma~\ref{lemmapolypow} it follows by $\Sigma^B_0(\card)$-induction that exponentiation by squaring is equivalent to first taking true exponentiation and then compute the remainder modulo $\X^r-1$.
\end{proof}

\subsection{Division of high degree polynomials and provability of $\DLB$}

It this section, we will show that the division of high degree polynomials over a bounded field is total in $\VTC^0_2$, and thus we can formalize the usual proof of the $\DLB$ axiom. It is straightforward to find a $\Sigma^B_0(\card)$ formula defining a predicate $P\in F[\X]$ formalizing that $P$ is a high degree polynomial over the bounded field $F$.

\begin{lemma}[$\VTC^0_2$]\label{lemmafieldpolypow}
    Let $F$ be a bounded field, then there is a $\Sigma^B_1$-definable function which to each high degree polynomial $P\in F[\X]$ and a number $x$ assigns a high degree polynomial $P^x$ and $\VTC^0_2$ proves for every $x$ and every high degree polynomial $P\in F[\X]$:
    \begin{align*}
        P^0 &= 1 \\
        P^{x+1} &= P \cdot P^x
    \end{align*}
\end{lemma}
\begin{proof}
Can be proven analogously to Lemma~\ref{lemmapolypow}.
\end{proof}

\begin{lemma}[$\VTC^0_2$, adapted from~\cite{healy2006fieldops}]\label{lemmavtcpolydiv}
     For every bounded field $F$ and every $P,S\in F[\X]$ there are $Q,R \in F[\X]$ such that \[\deg(R)< \deg(S)\text{ and }P=S\cdot Q + R.\]
\end{lemma}
\begin{proof}
    Without loss of generality, we can assume $P$ and $S$ to be monic. Further let $\deg P = n$, $\deg S = m$, $P=\sum_{i=0}^n a_i \X^i$ and $S=\sum_{i=0}^m b_i \X^i$. Define $S_\R=\sum_{i=0}^n a_{n-i}\X^i$ and $S_\R=\sum_{i=0}^m b_{m-i}\X ^i$. By Lemma~\ref{lemmapolypow} and the totality of iterated addition we can define $\tilde S_\R = \sum_{i=0}^{n-m} (1-S_\R)^i$.

    Then put $H=P_\R\tilde S_\R$, assume $H=\sum_{i=0}^{2n+m} c_i \X^i$. Finally, define the polynomials $Q=\sum_{i=0}^{n-m} c_{n-m-i}\X^i$ and $R=f-g\cdot q$.

    Now it just remains to show that $\deg R < m$. First notice that
    \begin{align*}
        \tilde S_\R S_\R &= \tilde S_\R(1-(1-S_\R)) \\
                       &= \tilde S_\R- \tilde S_\R(1-S_\R)\\
                       &=\left (\sum_{i=0}^{n-m}(1-g_R)^i \right) - \left(\sum_{i=0}^{n-m}(1-g_R)^{i+1}\right)\\
                       &= 1-(1-S_\R)^{n-m+1}.
    \end{align*}
    Since $S$ is monic, $1-S_\R$ has no constant term and so $(1-S_R)^{n-m+1}$ has its lowest $n-m+1$ coefficients with value $0$, in other words there exists $T \in F[\X]$ such that $(1-S_\R)^{n-m+1}=\X^{n-m+1}T$. Let $d=\deg R$, assume for contradiction, that $d\geq m$. Define $Q_\R = \X^{n-m}Q(1/\X)$, $R_\R = \X^d R(1/\X)$. Then
    \begin{align*}
        P &= S\cdot Q + R\\
        P(1/\X) &= S(1/\X)\cdot Q(1/\X) + R(1/\X)\\
        \X^n P(1/\X) &= (\X^m S(1/\X)) \cdot (\X^{n-m} Q(1/\X)) + \X^{n-d} (\X^d R(1/\X))\\
        P_\R &= S_\R \cdot Q_\R + x^{n-d}R_\R.
    \end{align*}
    Now since $H= P_\R \tilde S_\R = (S_\R\cdot Q_\R + x^{n-d}R_\R)\tilde S_\R = Q_\R (1-(1-S_\R)^{n-m+1})+\X^{n-d} \tilde S_\R R_\R$ then $H= Q_\R - \X^{n-m+1}T + \X^{n-d}\tilde S_\R R_\R$. But from the fact that $n-d < n-m+1$ and $\tilde S_\R$ and $R_\R$ have their constant coefficients equal to $1$ we obtain that $\X^{n-d}\tilde S_\R R_\R$ has the coefficient of the monomial $\X^{n-d}$ non-zero and thus $Q_\R$ is not equal to the lowest $n-m+1$ coefficients of $H$, a contradiction.
\end{proof}

The following theorem shows provability of $\DLB$ in $\VTC^0_2$. In the original formulation of $\DLB$, sparse polynomial are used but every such polynomial determines a high degree polynomial, so the statement we give is actually at least as strong as the original one.

\begin{theorem}[The theory $\VTC^0_2$ proves $\DLB$.]\label{thrmdlb}
    There is a $\Sigma^B_1$-definable function $I_{F,G}(x)$ such that $\VTC^0_2$ proves: 
    
    For every bounded field $F$ and a high degree polynomial $G\in k[\X]$ the function \[I_{F,G}(x):\{a\in F ; G(a)=0\} \to \{1,\dots,\deg G\},\]
    is injective.
\end{theorem}
\begin{proof}
    There is a $\Sigma^B_1$-definable function $D(F,G,x)$ such that for any bounded field $F$ with bound $b$ on its universe and a high degree polynomial $G\in F[\X]$ we have for all $m\leq b$: $D_{F,G}(x)(m)=\prod_{i\in\{0,\dots,m-1\},G(i)=0} (\X-i)$. 

    \textbf{Claim:} For all $m\leq b$:  $D_{F,G}(m)\mid G$.
    
    \emph{Proof of claim.} By induction on $m$. For $m=0$, we have $D_{F,G}(0)=1$ and thus $D_{F,G}(0) \mid G$. Assume the statement holds for $m$, that is $D_{F,G}(m)$ divides $G$. If $G(m)\neq 0$, then $D_{F,G}(m+1)=D_{F,G}(k) \mid G$ and we are done. Therefore, we assume that $G(m)=0$. By the induction hypothesis, there is $H\in k[\X]$ such that $G = D_{F,G}(m) \cdot H$. Hence, $0=G(m)=D_{F,G}(m)\cdot H(m)$. And since $D_{F,G}(m) = \prod_{i\in\{0,\dots, m-1\},G(i)=0}(m-i)$ is a product of nonzero elements of $k$, we have that $H(m)=0$. By Lemma~\ref{lemmavtcpolydiv}, there are $Q$ and $R$ such that $H=(\X-m)Q+R$, if $R\neq 0$, then $0=H(0) = 0\cdot Q + R \neq 0$. This implies that $R=0$ and therefore $(\X-m)\cdot Q = H$. Together we get that $D_{F,G}(k+1)Q = D_{F,G}(m) (\X-m) Q = G_k H = G$ and that $D_{F,G}(m+1)\mid G$ which proves the claim.

    Finally, we can $\Sigma^B_1$-define a function $I_{F,G}(x)$ such that for any bounded field $F$ with bound $b$ on its universe and a high degree polynomial $G\in k[\X]$ we have for all $m\in\{0,\dots, b\}:$
    \[I_{F,G}(m) = 
    \begin{cases}
        0 & m=0,\\
        I_{F,G}(m-1)+1 & \text{else if }D_{F,G}(m-1)\neq D_{F,G}(m),\\
        I_{F,G}(m-1) & \text{otherwise.}\\
    \end{cases}
    \]
    By induction on $m$ we get that for all $m\leq b$ we have \[I_{F,G}(m) = \deg D_{F,G}(m) \leq \deg G\] and that $I_{F,G}(-)$ is non-decreasing. Let $a,b\in k$ such that $G(a)=G(b)=0$ and $a<b$, then $I_{F,G}(a)<I_{F,G}(b)$, which shows the injectivity of $I_{F,G}(x)$.
\end{proof}

Finally, we have everything needed to finish the proof of correctness in the theory $\VTC^0_2$.

\begin{theorem} \label{Thm: VTC0}
        $\VTC^0_2$ proves the consequences of the theory $S^1_2+\iWPHP+\DLB+\GFLT$ restricted to the language of $\PV$, that is, the consequences without the symbol $\iota$.
\end{theorem}
\begin{proof}
    By Theorem~\ref{thrmdlb}, the theory $\VTC^0_2$ proves that a $\Sigma^B_1$-definition of $I_{F,G}(x)$ satisfies the axiom $\DLB(\iota)$ with $\iota$ replaced by the definable function $I_{F,G}(x)$. By the $\Sigma^B_0(\card)\text{-}\mathrm{COMP}$, we can show that it it also proves the $\sblm$ scheme with the symbol $\iota$ being replaced by the definition of $I_{F,G}(x)$.

    It is well known, that $\VTC^0$ proves the $\PHP$ for maps represented by an object of the set sort, by the $\Sigma^B_0(\card)\text{-}\mathrm{COMP}$ scheme and the $\Sigma^B_1$-definability of $I_{F,G}(x)$ we obtain that it also proves $S^1_2 + \iWPHP(\iota)+ \DLB$ with $\iota$ replaced by the definable map $I_{F,G}(x)$. The remaining $\GFLT$ axiom is provable in $\VTC^0_2$ by Lemma~\ref{lemmabinomdiv}.
\end{proof}

Using Theorem~\ref{Thm: VTC0} and Theorem~\ref{theoremT2proof} we obtain the following.

\begin{corollary}
    The theory $\VTC^0_2$, and consequently the theory $T^\ct_2$, both prove the sentence $\AKSCorrect$.
\end{corollary}
\section{Concluding remarks}

In this work, we proved the correctness of the AKS algorithm as a first order sentence in the language of $\PV$ in the theory $\VTC^0_2$. Obtaining such a proof inside $T_2$ seems to be unlikely as even ordinary Fermat's Little Theorem implies, over $T_2$, an instance of $\PHP$ which is not known to be provable in $T_2$~\cite{EmilAbelian}. Therefore, we believe it is natural to ask about the following problem.

\begin{problem}
    Does $T_2 + \PHP(\Sigma^b_\infty)$ prove the correctness of the AKS algorithm?    
\end{problem}

We do know, by Jeřábek's work~\cite{EmilAbelian}, that $T_2+\PHP$ proves Fermat's Little Theorem. Does this imply anything about $\GFLT$ in $T_2+\PHP$?

\begin{problem}
    Does $T_2 + \PHP(\Sigma^b_\infty)$ prove the $\GFLT$ axiom?
\end{problem}

Showing the unprovability of either of the above statements under some complexity theoretic assumptions is just as interesting. In fact, it is an old problem of Macintyre whether $I\Delta_0$ proves $\PHP$ for bounded formulas. It is also tempting to understand the role of the function computed by the polynomial $(\X + a)^p$, or any similar function, in the context of proof complexity generators. 

The provability of the implication $\Prime(x)\to \AKSPrime(x)$, which is a $\Sigma^b_1$-formula, relates to the complexity of total $\NP$ search problems by well known witnessing theorems. The relevant total problem is: For an input number $x$ either verify it is prime using the AKS algorithm or find a proper divisor. We show provability of this implication in the theory $S^1_2 + \GFLT$. Unfortunately, it seems the corresponding witnessing theorem gives only trivial reductions for the factorization problem. It would be interesting to find a provability of this implication in a theory which gives non-trivial witnessing.

The other implication, $\AKSPrime(x)\to\Prime(x)$, is a $\Pi^b_1$ statement and thus relates to propositional proof complexity by results about propositional translations. Our proof in $\VTC^0_2$ implies a proof in $U^1_2$ whose corresponding proof system is the quantified sequent calculus $G$. The theory $\VTC^0_2$ also proves everything $T_2$ does, which implies that the proof we obtain by propositional translations should be in a system above every $G_i$, but below $G$.

\begin{problem}
    Describe the propositional proof system corresponding to the $\Pi^b_1$-consequences of the theory $\VTC^0_2$ or alternatively to the theory $T^\ct_2$.
\end{problem}

The theory $S^1_2+\DLB+\iWPHP+\GFLT$ is likely weaker than full $T^\ct_2$. It could also be interesting to find the propositional proof system corresponding to its $\Pi^b_1$-consequences.

\section*{Acknowledgement}
%We would like to thank Jan Kraj{\'\i}{\v{c}}ek, Emil Je{\v{r}}{\'a}bek, and Pavel Pudl{\'a}k for fruitful discussions. \\
We are deeply grateful to Emil Je{\v{r}}{\'a}bek for his generous and thoughtful feedback on an earlier draft of this work. His comments helped us clarify key arguments, correct important inaccuracies, and his suggestion to consult~\cite{healy2006fieldops} was especially valuable. We are grateful to Jan Kraj{\'\i}{\v{c}}ek and Pavel Pudl{\'a}k for their comments and guidance. We also thank Eitetsu Ken, Erfan Khaniki, Faruk G\"{o}loglu, and Amir Tabatabai for helpful discussions.

%The authors are indepted to Jan Krajíček and Pavel Pudlák for their comments and guidance. We are also especially grateful to Emil Jeřábek for pointing us to~\cite{healy2006fieldops}, and for his valuable comments on a draft of this work, which helped us refine our arguments and correct certain inaccuracies. We thank Eitetsu Ken, Erfan Khaniki and Faruk Göloglu for helpful discussions.
\bibliographystyle{plain}
\bibliography{Primes}

\end{document}